\documentclass[12pt,leqno]{article}
\usepackage{geometry}
\usepackage{amsthm,amsfonts,amssymb,epsfig,graphics,amsmath}
\usepackage{mathrsfs}
\usepackage{epsfig}
\usepackage{mathabx}
\usepackage[utf8]{inputenc} 
\usepackage{authblk}
\usepackage{hyperref}

\newcommand{\Cc}{{\mathscr C}}
\newcommand{\Fc}{{\mathscr F}}

\newcommand{\RR}{{\mathbb R}}
\newcommand{\NN}{{\mathbb N}}
\newcommand{\ZZ}{{\mathbb Z}}

\newcommand{\CC}{{\mathbb C}}

\newcommand\cA{{\mathcal A}} 
\newcommand\cB{{\mathcal B}} 
\newcommand\cC{{\mathcal C}}
\newcommand\cD{{\mathcal D}}

\newcommand\cS{{\mathcal S}}
\newcommand\cF{{\mathcal F}}
\newcommand\cM{{\mathcal M}}

\newcommand\cN{{\mathcal N}}

\def\eps{\varepsilon}
\newcommand\vp{\varphi} 
\def\D{\partial}
\def\wh{\widehat}

\def\ud{\underline}
\def\ds{\displaystyle}

\renewcommand{\div}{{\rm div}}

\newcommand{\na}{{\nabla}}

\newtheorem{theo}{Theorem}[section]
\newtheorem{lem}[theo]{Lemma}
\newtheorem{defi}[theo]{Definition}
\newtheorem{ass}[theo]{Assumption}

\newtheorem{rem}[theo]{Remark}

\title{Asymptotic expansion of wave scattering in a periodic 2d-plane\footnote{The following article has been submitted to AIP Journal of mathematical physics. After publication if any, it would be found at \url{https://publishing.aip.org/resources/librarians/products/journals/}}
}
\author{Vincent Lescarret\thanks{vincent.lescarret@centralesupelec.fr}}
\affil{Université Paris Saclay,\\ Fédération de mathématiques de Centralesupelec, CNRS FR3487}
\begin{document}

\maketitle

\begin{abstract}
We give a counter part of Sommerfeld outging radiation condition for waves propagating in a 2d periodic medium under generical assumptions and provide a uniqueness theorem for outgoing solutions.
\end{abstract}

\section{Introduction}
Asymptotics of the outgoing Green function for the Helmholtz equation with periodic coefficients has been given in~\cite{MT} for frequencies lying in the first spectral band in any dimensions. We propose to extend the formula in the 2d case to any frequency except for a set of isolated frequencies. Part of this work reproduces the work~\cite{SAKM} which the author just became aware by the time of submission of this present work. Despite redundance of some ideas this work adresses many other points and partly lies on~\cite{G}.

We wish to solve the Helmholtz equation
\begin{equation}\label{e1}
\div(\alpha\na u)+k^2\beta u=f\quad \text{in}\ \RR^2
\end{equation}
where $f\in L^2(\RR^2)$ is compactly supported, $k$ real positive and $\alpha,\beta>0$. The coefficients $\alpha$ and $\beta$ are bounded functions, periodic with common period. Let $W=[0,2\pi]^2$ be a periodicity cell (Wigner-Seitz cell) and $B=[0,1]^2$ the fundamental periodicity cell of the reciprocical lattice (Brillouin zone).

We look for a formula for $u$ by the mean of the absorption principle as in~\cite{MT,Ra}. A general formula has been given in~\cite{Ra}, Theorem 3.31 expressing $u=u_1+u_2+r$ where $r\in L^2(\RR^2)$, $u_j\in L^2_{loc}$ for $j\in\{1,2\}$.
Then $u_1$ is a residu while $u_2$ is a principal Cauchy value. In Remark 3.33 of~\cite{Ra} the author says that $u_1$ is the leading term and that $u_2$ is a corrector in some $L^p$ space but Lemma 2.4 of~\cite{MT} shows that this is wrong.
Loosely speaking the term $u_2$ removes the terms in $u_1$ which correspond to ``incoming'' waves and thus $u_1+u_2$ only keeps ``outgoing'' waves (see~\cite{FJ15} for the idea in the case of a periodic waveguide). 
 
Our calculations closely follow those of~\cite{MT} (which are based on the method used in the homogeneous (non periodic) case as for instance in Melrose~\cite{Mel} paragraph 1.7) but with two differences. First instead of proving an analogue of Lemma 2.4 of~\cite{MT} we just use Cauchy residu formula and thus deal with contour complex deformation as in~\cite{Ho,G}. Then contrary to~\cite{MT} we consider any $k^2$ above the bottom of the essential spectrum of $P=-\frac{1}{\beta}\div(\alpha\na)$. In~\cite{MT}, $k^2$ lies on the first band and is close enough to the bottom of the essential spectrum of $P$ so that the level set on the first band is a single smooth cycle (loop). Here generally several bands meet the level $k^2$ and one requires a refined analysis of the geometry of the level set. Of course ideas are known for a long time in crystallography and in elasticity where level sets are called slowness surfaces (in 3d). See for instance~\cite{Wil1,BGJ}.

Before stating the main result let us start with a formal calculation and introduce the main notations.

Let $\eps>0$ and replace $k^2$ in~\eqref{e1} by $k_{\eps}^2:=k^2+i\eps$. Then applying the Bloch transform:
$$
\hat u(x,\ell)=\sum_{j\in\ZZ^2}u(x+2\pi j)e^{-i\ell\cdot (x+2\pi j)}
$$
to equation~\eqref{e1} and using the commutational property of the Bloch transform (see~\cite{AlCoVa}) we get
$$
(\div+i\ell\cdot)\left (\alpha(\na+i\ell) \hat u\right )+k_{\eps}^2\beta\hat u=\hat f\quad \text{in}\ W,
$$
together with periodic boundary conditions on $W$.
When $\alpha$ is piecewise continuous the underlying operator $P(\ell):=-\beta^{-1}(\div+i\ell)\alpha(\na+i\ell)$ is defined as the m-sectorial operator (see~\cite{K}) associated to the sectorial sesquilinear form 
\begin{align*}
&\int_W\alpha(\na+i\ell)u(\na-i\ell)\bar v dx\quad\text{on}\\
& H^1_{per}(W)=\{v\in H^1(W,\beta dx),\ \text{with periodic boundary conditions on}\ \D W\}
\end{align*}
We readily see that $P(\ell)$ is symmetric: $P(\ell)^*=P(\bar\ell)$ and it is well-known~\cite{RS} that it has a discrete spectrum which we denote by $\{\lambda_n(\ell)\}_{n>0}$ (counting multiplicity) with $\lambda_n(\ell)$ real for real $\ell$ and we denote by $e_n(\ell)$ the corresponding eigenvectors.
Since $P(\ell)$ is defined through a sesquilinear form the family $P(\ell)$ is analytic of type B (see~\cite{K}, § 4.2, p.393) thus the functions $\ell\rightarrow \lambda_n(\ell)$ are piecewise analytic and continuous on $\CC^2$. Besides, the Bloch variety
$$
\cB=\{(\ell,\lambda)\in\CC^3\ \text{such that}\ \exists n,\ \lambda_n(\ell)=\lambda \}
$$
is an analytic set because it is the null set of a regularized determinant which is an entire function, see Appendix~\ref{A2} where we recall and adapt~\cite{Ku} in this more general setting. See Figure~\ref{fig1} illustrating the Bloch variety.

Expanding the solution in the Hilbert basis $\{e_n\}_{n>0}$ one has
\begin{equation}\label{s1}
u_{\eps}(x)=\int_Be^{i\ell\cdot x}\sum_{n\in\NN^*}\frac{(\hat f,e_n(\ell))_{L^2(W)}}{k_{\eps}^2-\lambda_n(\ell)}e_n(x,\ell)d\ell.
\end{equation}
Only a finite number of terms in the sum have singular limit when $\eps$ goes to zero. We thus set
\begin{equation}\label{activeband}
J_{k}=\{n,\ \exists \ell\in B,\ \lambda_n(\ell)=k^2\}.
\end{equation}
Let us introduce the main following geometrical objects:
\begin{itemize}
\item For complex $\lambda$ let us set 
$$
\cF_{\lambda}=\{\ell\in \CC^2, \exists n, \lambda_n(\ell)=\lambda\}=\bigcup_{n>0}\lambda_n^{-1}(\lambda).
$$
which we refer to as the (complex) $\lambda$-Fermi level. 
\item For $\lambda>0$ we denote by $F_\lambda$ the set $\cF_\lambda\cap\RR^2$ which is periodic with $B$ a periodicity cell.
\item For $\lambda>0$ we put $F_\lambda^0=F_\lambda\cap B$.
\end{itemize}
We give an example of $F_{k^2}^0$ in Figure~\ref{fig1}.
\begin{figure}
\includegraphics[clip,width=8cm,height=7cm]{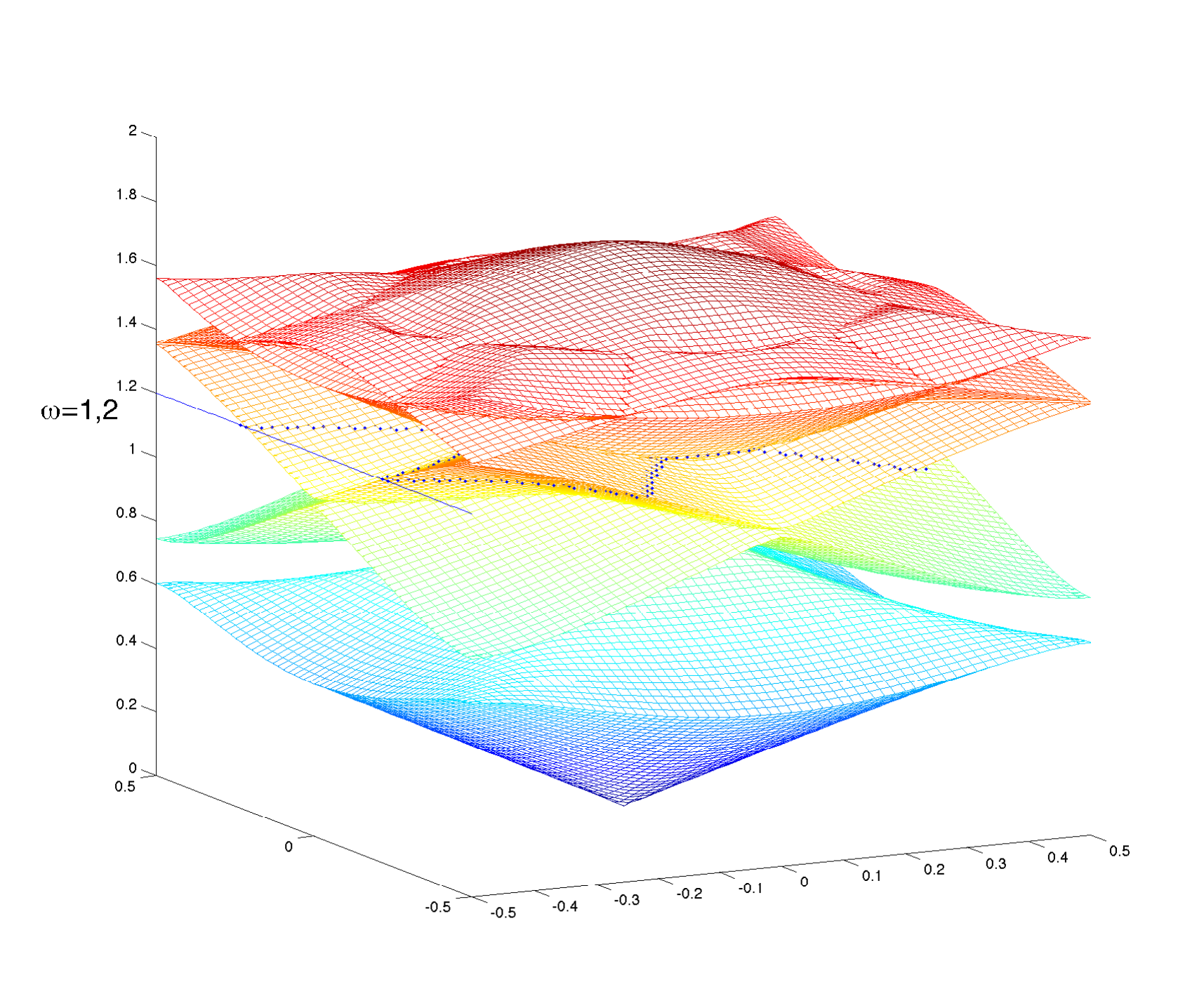}~
\includegraphics[clip,width=8cm,height=7cm]{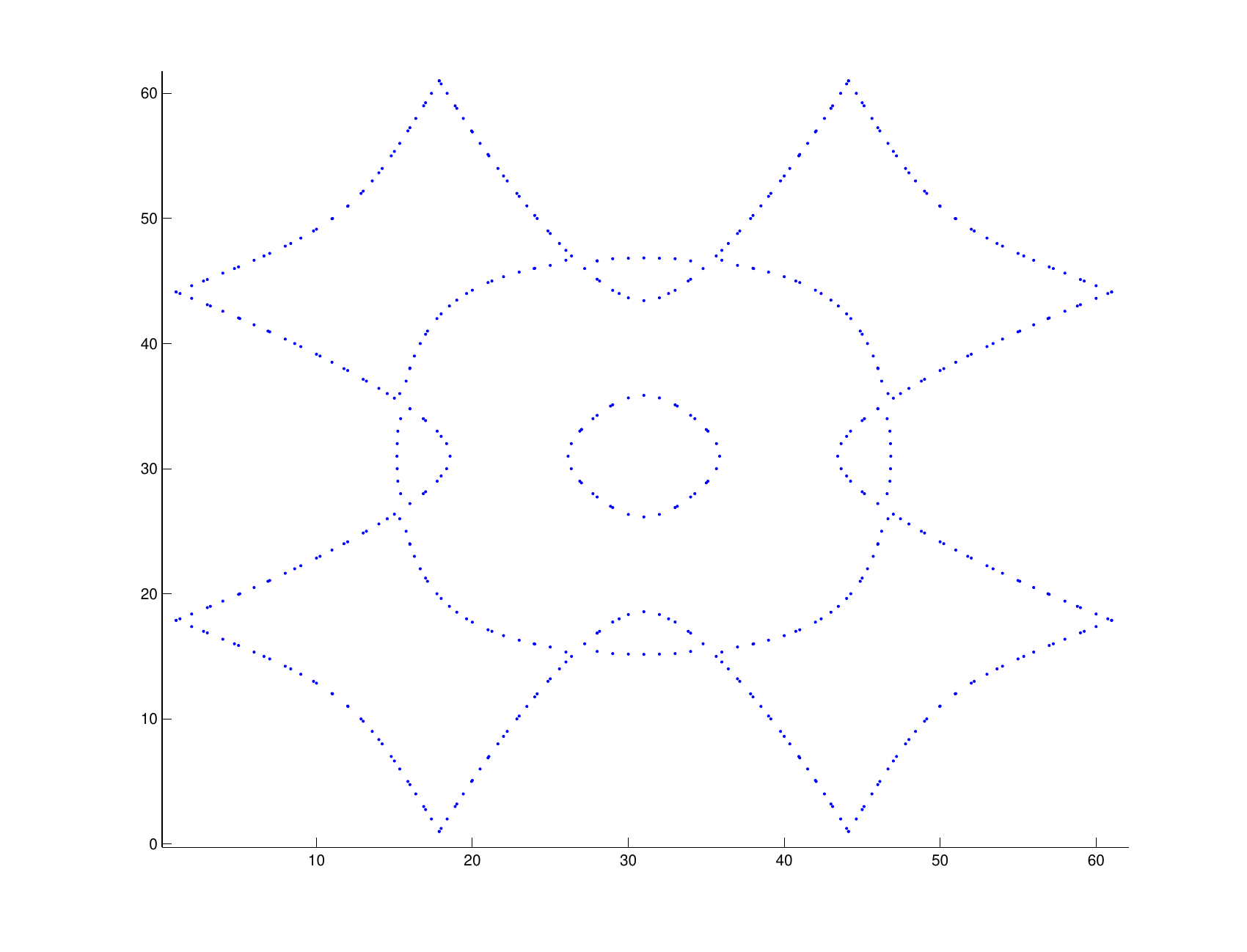}
\caption{Bloch variety (left) and Fermi level $F_{k^2}^0$ (right) with $k=1.2$ for $\alpha=1+0.8\cos(x)\cos(y)$ and $\beta=1$.}\label{fig1}
\end{figure}
With these notations we see that the terms in formula~\eqref{s1} whose index belong to $J_{k}$ are singular on $F_{k^2}^0$ for $\eps=0$. As in~\cite{Mel,Ho} our aim to handle those singular terms is to move $B$ to $\CC^2$ (actually a subset of $B$ which is a tubular neighborhood of $F_{k^2}^0$) and use the residu formula. 
So we need to find a complex deformation of $B$ avoiding the complex Fermi level $\cF_{k_{\eps}^2}$.
To do so we describe a tubular neighborhood of $F_{k^2}^0$ as a union of level sets. These level sets are first indexed by $\lambda$ in a small open intervall $I$ containing $k^2$. Then we deform $I$ to a complex curve avoiding $k_{\eps}^2$.

This procedure can be done for all $k^2$ except for the critical values of the band functions $\lambda_n$ and a subset of points of multiple eigenvalues (band crossing) which we call the set of \emph{singular crossing points} (see section~\ref{sgeo}, definition~\ref{def1}). The latter set is defined as the complementary set of \emph{regular crossing points}
characterized by the fact that up to index relabelling, the (two or more) functions $\lambda_j$ can be continued analytically through the crossing.
In the 1d real case band crossings are regular thanks to Rellich eigenvalue relabelling theorem~\cite{Re2}. In the several dimensional real case this is true~\cite{KP} except for singular crossing points.
We thus exclude the following sets
\begin{itemize}
\item the set $\sigma_0$ of real critical values of the family $\{\lambda_n\}$,
\item the set $\sigma_1=\{\lambda\in\RR,\ \exists\ell\in B\ |\ (\ell,\lambda)\ \text{is a singular crossing point}\}$.
\end{itemize}
That those sets are made of isolated points is a consequence of the stratified structure of the Bloch manifold and the dimension 2 (see Section~\ref{sgeo}).

The set $\sigma_0$ is called the set of ''Landau resonences'' in~\cite{G} where it is shown to be made of isolated points.
In~\cite{G} an other set denoted by $\sigma_{\infty}$ is also avoided but it matters only when one considers the global holomorphic extention of the resolvent operator $(P-zI)^{-1}$ from Im$z>0$ to a complex neighborhood of $\RR$ as an operator from $L^2_{comp}$ to $L^2_{loc}$. In this paper we are not concerned with $\sigma_{\infty}$ since we consider the resolvent in a small neighborhood of $\RR$ only.
It is shown in~\cite{G} that any point of $\sigma_0$ is a branch point for the resolvent associated to the equation~\eqref{e1}. See section~\ref{Sglancing} where we recall the related expression of the resolvent in the neighborhood of $\sigma_0$ in this 2d case. Let us remark that~\cite{G} does not address the issue of the assymptotic expansion of the resolvent but only its regularity (holomorphy).

A direct consequence of~\cite{KP}, Theorem 6.7 is that $\sigma_1$ is a set of isolated points of the real Bloch variety and the tangent set of such a point is a (non-isotropic) cone which is not a cusp. This is typically the case of Dirac cone~\cite{Wang}. Let us already say that the subsequent analysis takes advantage of the fact that this set is made of isolated points and thus one needs to implement a missing step to deal with higher dimensions where points of $\sigma_1$ are generically non isolated. 



\section{Main results}

Our first result is a generic formula for the leading part of the limit of $u_{\eps}$ as $\eps$ goes to zero. Generically $F_{k^2}^0$ is 1-dimensional or void. We address the former case (the latter is already well-known and scattering does not take place). The formula we get is an integral on $F_{k^2}^0$ corresponding to the limit of the residu of the expression~\eqref{s1}. The expression involves the spectral projector of $P(\ell)$ which is in general position one dimensional and given by $(\hat f,e_n(\ell))_{L^2(W)}e_n(x,\ell)$ for $\ell\in\lambda_n^{-1}(k^2)$. This expression is false when $\lambda_n(\ell)$ is multiple.

In general position the set $\cC_{k^2}$ of points in $F_{k^2}^0$ at which two or more bands cross is generically finite.
Thus for $\ell\in F_{k^2}^0\setminus\cC_{k^2}$ there is a unique integer $n\in J_{k}$ such that $\ell\in\lambda_n^{-1}(k^2)$ hence one can define the following two functions a.e. in $F_{k^2}^0$
\begin{equation}\label{pw}
\tilde\lambda(\ell):=\lambda_n(\ell)\quad\text{and}\quad \tilde e(\ell):=e_n(\ell).
\end{equation}

Our first result is
\begin{theo}\label{t1}
Let $k$ be such that $k^2\notin\sigma_0\cup\sigma_1$ in general position. Let
$$
F_{k^2}^+(x)=\{\ell\in F_{k^2}^0\setminus\cC_{k^2},\ \na\tilde\lambda(\ell)\cdot x>0\}.
$$
Then $u_{\eps}$ converges to $u$ in $H^1_{loc}$ expanding
\begin{equation}\label{f1}
u(x)=2i\pi\int_{F_{k^2}^+(x)}e^{i\ell\cdot x}\frac{(\hat f(\ell), \tilde e(\ell))_{L^2(W)}}{\|\na\tilde\lambda(\ell)\|}\tilde e(x,\ell)\sigma(d\ell)+R(x)
\end{equation}
where $\sigma$ is the length measure on $F_{k^2}^+(x)$ and $R\in H^1(\RR^2)$.
\end{theo}
See Figure~\ref{fig2} for an example of $F_{k^2}^+(x)$.
\begin{figure}
\begin{center}
\includegraphics[clip,width=10cm,height=9cm]{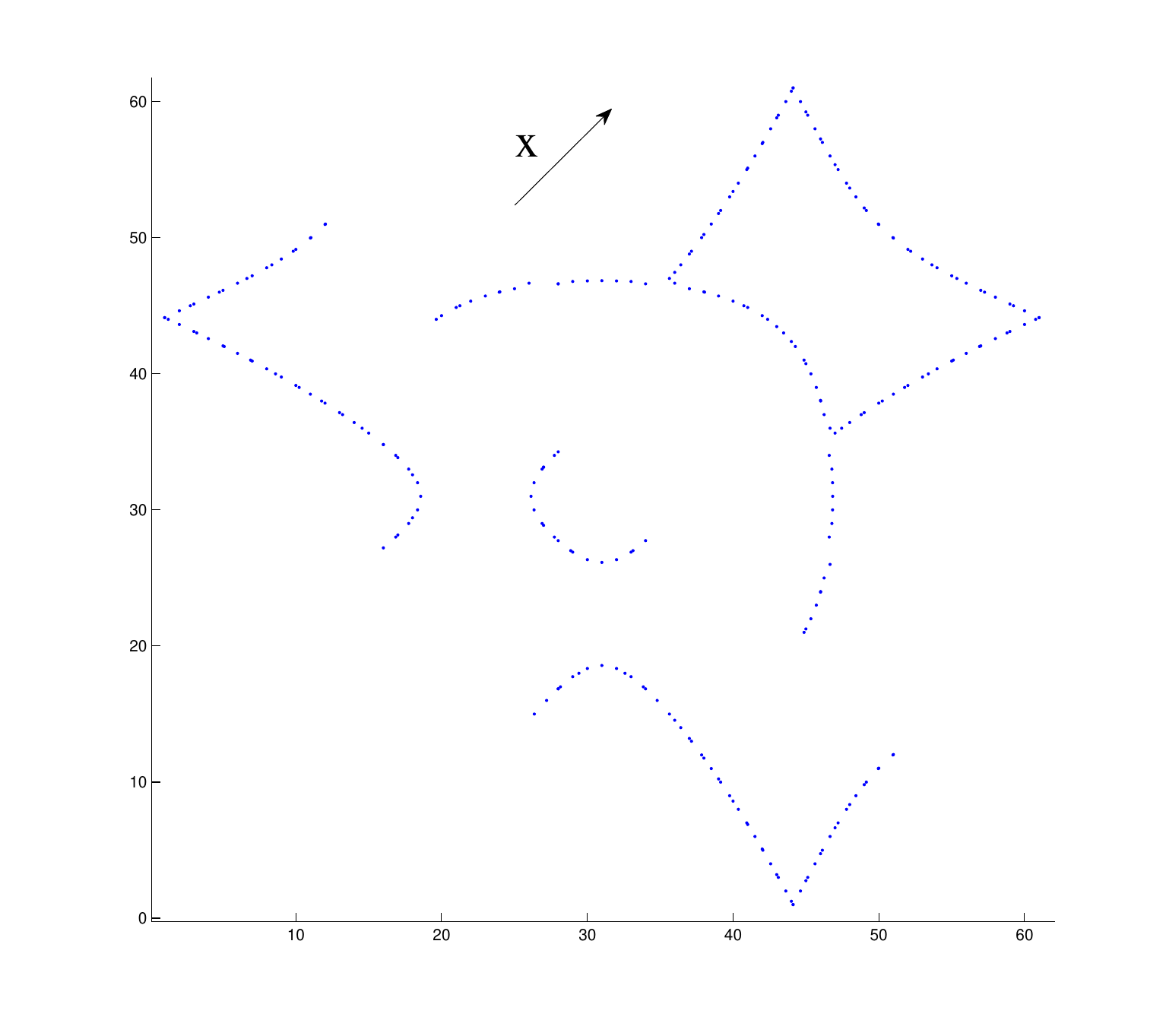}
\caption{The set $F_{k^2}^+(x)$ with $x=(1,1)$ and $\alpha,k$ as in Figure~\ref{fig1}.}\label{fig2}
\end{center}
\end{figure}

Since the integrand is quasiperiodic one can provide a full expansion of the integral as a series of fractional powers of $1/|x|$ by the mean of the stationary phase method. However a difficulty arises because $F^0_{k^2}$ is generically not convex. Using the periodicity of the integrand with respect to the Floquet variable $\ell$ one can arrange $F^0_{k^2}$ as the union of close smooth curves or periodic smooth curves (see Figure~\ref{fig3}). Some curves are convex others have inflexion points $\ud\ell_j$. 

The stationary phase method shows that critical points of the phase $i\ell\cdot x$ are such that $\na\tilde\lambda$ is parallell to $x$ at such points. Let us denote them by $\ell_n$. Then $\ell_n$ is a function of $\theta=arg(x)$ and is well defined as long as it does not meet any inflexion point. See Figure~\ref{fig4}.
\begin{figure}
\includegraphics[clip,width=16cm,height=7cm]{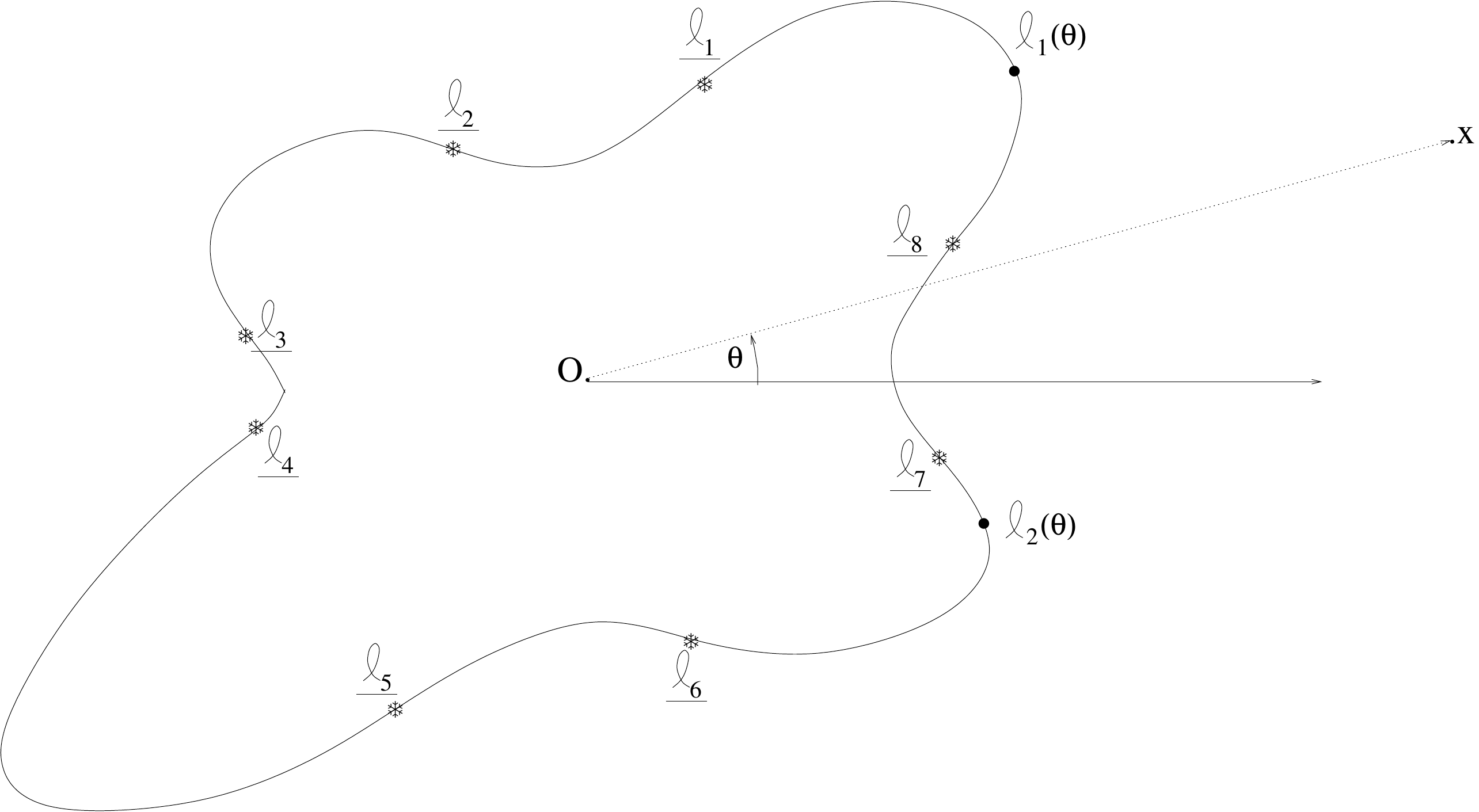}
\caption{A closed component of $F_{k^2}^0$ with the points $\ud\ell_j$ and $\ell_n(\theta)\in F_{k^2}^+(x)$}\label{fig4}
\end{figure}
Again for sake of simplicity we adress the problem in general position assuming that
inflexion points are \emph{non degenerated} and two inflexion points correspond to \emph{distinct} angles $\theta_j$.
Let us thus denote by $I_n\subset[0,2\pi]$ the domain of $\ell_n$ and denote by $\theta_j$ the angle for which for some $n$, $\ell_n(\theta_j)=\ud\ell_j$. Finally let $\cN_j$ be a small neighborhood of $\theta_j$ which does not contain any other inflexion point.

Then the leading term in the expression of $u$ expands asymptotically according to

\begin{theo}\label{t2}
  Let $k$ as in Theorem~\ref{t1}.

  1) For $\theta\in[0,2\pi]\setminus(\cup_j\cN_j)$ there holds for $r=\|x\|$ big
\begin{equation}\label{exp}
u(x)=\frac{i\sqrt{2\pi}}{\sqrt{r}} e^{-i\pi/4}\sum_{n=1}1_{I_n}(\theta) e^{i\ell_n(\theta)\cdot x}\frac{(\hat f(\ell_n),\tilde e(\ell_n))_{L^2(W)}}{\|\na\tilde\lambda(\ell_n)\|\kappa_n^{1/2}}\tilde e(x,\ell_n)+\tilde R,
\end{equation}
where $\kappa_n$ is the curvature of $F^+_{k^2}(x)$ at $\ell_n$ and $\tilde R\in H^1(\RR^2)$.

2) For $\theta\in\cN_j$ with $\theta_j=\bar I_{n_1}\cap \bar I_{n_2}$
\begin{multline}
  u(x)=\frac{e^{i\ud\ell_j\cdot x}}{r^{1/3}}\Big(\alpha_jAi(r^{2/3}\gamma_j(\theta-\theta_j))\tilde e(\ud\ell_j,x)+\frac{\beta_j}{r^{1/3}}Ai'(r^{2/3}\gamma_j(\theta-\theta_j)) w_j(x)\Big)\label{fAi}\\
  +\frac{i\sqrt{2\pi}}{\sqrt{r}} e^{-i\pi/4}\sum_{n\notin\{n_1, n_2\}}e^{i\ell_n\cdot x}\frac{(\hat f(\ell_n),\tilde e(\ell_n))_{L^2(W)}}{\|\na\tilde\lambda(\ell_n)\|\kappa_n^{1/2}}\tilde e(x,\ell_n)+\tilde R
\end{multline}
where $Ai$ is Airy's function, $a_j,b_j$ are non zero real, $\gamma_j(0)=0$, $w_j$ is periodic and belongs to $H^1(W)$ and $\tilde R\in H^1(\RR^2)$. 
\end{theo}
\begin{rem}
In~\eqref{fAi} the decay rate of the first term is $O(r^{-1/2})$ because the oscillatory part of Airy function $Ai(s)$ decays like $O(s^{-1/4})$.
  
We expect that the remainder decreases more rapidely in the far. Actually $\tilde R=R_1+R_2$ where $R_1$ corresponds to terms in $F_{k^2}^-=F_{k^2}^0-F_{k^2}^+$ and $R_2$ is related to the stationary phase theorem. We prove that $R_1$ decreases faster than any polynomial and $R_2=O(r^{-3/2})$. However we don't know how to prove any better decrease result for $R$ because it is a Bloch inverse transform whose integration set meets $\sigma_1$ if it is not empty. This prevents the use of the instationary phase theorem.
\end{rem}

\begin{defi}\label{defi1}
  A solution to equation~\eqref{e1} is called an outgoing solution if there are finitely many open intervals $(I_n)_n$ of $[0,2\pi]$ and a neighborhood $\cN$ of the boundaries $\theta_j$ of these intervals such that for $\theta\in [0,2\pi]\setminus\cN$
  \begin{equation}\label{convexp}
  u(x)=\frac{1}{\sqrt{r}}\sum_n1_{I_n}(\theta)c_n(\theta) e^{i\ell_n\cdot x}\tilde e(x,\ell_n)+R
  \end{equation}
  where $\ell_n\in F_{k^2}^+$ and $\na\tilde\lambda(\ell_n)$ is parallell to $x$ and $R\in H^1(\RR^2)$ while for $\theta\in\cN_j$ a small neighborhood of one extremity $\theta_j=I_{n_1}\cap I_{n_2}$
  \begin{equation}\label{inflexp}
u=\frac{e^{ix\cdot\ud\ell_j}}{r^{1/3}}\Big(\tilde c_jAi(r^{2/3}\gamma_j(\theta-\theta_j))\tilde e(\ud\ell_j,x)+\frac{1}{r^{1/3}}Ai'(r^{2/3}\gamma_j(\theta-\theta_j)) w_j(x)\Big)+\frac{1}{\sqrt{r}}\sum_{n\notin\{n_1, n_2\}}c_n(\theta)e^{i\ell_n\cdot x}\tilde e(x,\ell_n)+R
\end{equation}
where $\ud\ell_j=\ell_{n_1}(\theta_j)=\ell_{n_2}(\theta_j)$, $w_j$ is a periodic fonction belonging to $H^1(W)$ and $R\in H^1(\RR^2)$.
\end{defi}
Uniqueness of outgoing solutions requires that $P$ has no eigenvalue which is the case since the spectrum of $P$ is the union of $\lambda_n(B)$. However $P$ may have singular spectrum corresponding to the fact that one of the $\lambda_n$ is flat on a non empty ball.
\begin{ass}\label{ass1}
For all $n$, $\lambda_n$ is a non constant function on any open set.
\end{ass}
Under this assumption the spectrum of $P$ is purely essential~\cite{BS}.
Finally we need a technical assumption to prove uniqueness in Rellich's way:  we need that the remainder is smooth enough to consider the trace of $\na R$ along a circle. This is true if we assume
\begin{ass}\label{ass2}
The coefficient $\alpha$ is either lipschitz or discontinuous along smooth curves as in~\cite{LiNi} and in this case we also need $f\in L^{\infty}$.
\end{ass}
\begin{theo}\label{t3}
Under Assumptions~\ref{ass1} and \ref{ass2} equation~\eqref{e1} has a unique outgoing solution.  
\end{theo}



\section{Limiting absorption principle for the outgoing resolvent}

To prove Theorem~\ref{t1} we introduce a smooth cutoff function $\psi$ which vanishes everywhere except in a small neighborhood of $k^2$ on which the set $J_{k}$ remains constant. 

Then let us split
\begin{equation}\label{split1}
u_{\eps}=u_{1,\eps}+u_{2,\eps}\ ,\qquad u_{1,\eps}=\int_B\sum_{n\in J_k}\psi(\lambda_n(\ell))\frac{(\hat f,e_n(\ell))_{L^2(W)}}{k_{\eps}^2-\lambda_n(\ell)}e_n(x,\ell)d\ell.
\end{equation}
Let us first analyze $u_{2,\eps}=u-u_{1,\eps}$. By definition of $\psi$ there is a constant $c>0$ such that 
$$
\frac{1-\psi(\lambda_n(\ell))}{|k_{\eps}^2-\lambda_n(\ell)|}<c/\lambda_n(\ell)\quad\forall n
$$ 
so the Bloch transform
$$
\hat u_{2,\eps}(x,\ell)=\sum_{n\in \NN^*}(1-\psi(\lambda_n(\ell)))\frac{(\hat f,e_n(\ell))_{L^2(W)}}{k_{\eps}^2-\lambda_n(\ell)}e_n(x,\ell)
$$
belongs to $L^2(B;H^1(W))$. Thus $u_{2,\eps}\in H^1(\RR^2)$  (see Appendix~\ref{A1} where we collected some classical results about Floquet-Bloch transform on Sobolev spaces).

\bigbreak


To analyze and compute the limit of $u_{1,\eps}$ when $\eps$ goes to zero we need to modify the integration set $B$ to avoid $\cF_{k_{\eps}^2}$ for all positive $\eps$ close to zero. This was done by C. Gerard~\cite{G} in a theoretical way using a \emph{complex displacement} according to Pham~\cite{Ph}. Since the Fermi levels are parameterized by $\lambda$, a complex displacement amounts to choosing a homotopy for $\lambda$ from an interval around $k^2$ to a half loop in the lower complex plane. 
This allowed to extend the validity of the resolvant associated to~\eqref{e1} in a neighborhood of the real axis but no formula was given to compute the integral defining $u_{1,\eps}$ (except in the difficult case $k^2\in\sigma_0$).

Here on the contrary in order to compute the limit when $\eps$ goes to zero we push $\lambda$ to the upper complex plane over $k_{\eps}^2$ and use Cauchy residu formula.

Before going to the details we need more information about the topology of the level sets $F_{k^2}^0$ and explain how to continuously deform it to $\cF_{\lambda}$ when $\lambda$ goes to the complex domain. This is the aim of the next subsections.

\subsection{Geometry of a Fermi level}\label{sgeo}

Let us recall some general facts (see \cite{Wil}). Since the Bloch variety is an analytic set it possesses a Whitney stratification. This stratification is by regularity and dimension:


$$
\cB=\cB^r\cup\cB^\times,
$$
where $\cB^r$ is the regular part of $\cB$ which is open and locally a 2d-manifold and $\cB^\times$ the complementary set. The latter is a subset of the points where $\lambda_n$ are multiple. Indeed, by analytic perturbation theory, any point $(\ell,\lambda_n(\ell))$ where $\lambda_n$ is simple defines locally a manifold and thus is a regular point of $\cB$.
Again $\cB^{\times}=(\cB^{\times})^r\cup\cB^{\times\times}$ where $(\cB^{\times})^r$ is locally a 1d-manifold and any connected component of $\cB^r,(\cB^{\times})^r,\cB^{\times\times}$ is called a \emph{stratum}. It is a basic result from~\cite{Whi} that the number of strata is locally finite.

\begin{lem}\label{l1}
The set $\sigma_0$ is locally finite.
\end{lem}
\begin{proof}
Any stratum of $\cB^r$ is the graph of a unique analytic function $\lambda_n$ whose critical values are isolated by~\cite{So}. Any stratum $S$ of lower dimension is the graph of finitely many (generically two) crossing bands $\lambda_{j_1}=\ldots=\lambda_{j_n}$. Since $S$ is a manifold the restriction of $\lambda_k,\ k\in\{j_1,\ldots,j_n\}$ to the set $S$ is holomorphic and thus has at most a finite number of critical points. Finally 0 dimensional strata are isolated. Thus $\sigma_0$ is locally finite.
\end{proof}

We call \emph{singular stratum} a stratum of $\cB^{\times}$. Singular strata are in general position the sets of intersection of two bands $\lambda_{n_j},\ j=1,2$ which are simple outside the intersection set.

This is a finite dimensional problem for which we can use analyticity results about roots of hermitian matrices.

Let us proceed in details. First we consider the finite dimensional (matrix) reduction of $P(\ell)$ as follows. Since the spectrum of $P(\ell)$ is discrete and locally finite one can introduce the spectral projector on the finite dimensional vector space associated to a finite set of eigenvalues (cf. Kato~\cite{K} p.369 and 386). 
For $(\ell,\lambda_0)\in\cB$ with $\lambda_0$ a multiple eigenvalue, let us denote by $\pi(\ell,\lambda_0)$ the eigenprojection on the total eigenspace of $P(\ell)$ associated to the eigenvalue branches $\lambda_{n_j},\ j=1$ or $2$ in a neighborhood of $\lambda_0$. It reads
$$
\pi(\ell,\lambda_0)=\frac{1}{2i\pi}\int_{C(\lambda_0)}(P(\ell)-zI)^{-1}dz,
$$
where $C(\lambda_0)$ is a closed curve in the complex plane which encircles only $\lambda_0$ for $\ell=\ell_0$. Since $P(\ell)$ is an analytic familly of operators this projector is complex analytic on a small neighborhood of $\ell_0$.
Let us then set $T(\ell)=P(\ell)\pi(\ell,\lambda_0)$ which is a finite dimensional operator. Since $\pi$ is analytic $T$ is analytic too. 
Thus $T(\ell)$ reads as a $2\times 2$ hermitian matrix with complex analytic coefficients. We cannot use 1d Rellich's result~\cite{Re2} about the analytical continuation of eigenvalues of hermitian matrices. In this 2d case one needs to discuss the dimension of the crossing set
$$
\cM:=\{\lambda_{n_1}=\lambda_{n_2}\}\cap V_0
$$
in a neighborhood $V_0$ of $\ell_0$ (see~\cite{MeZu} paragraph 2.3 for a general discussion). Restricting ourself to real $\ell$ the authors in~\cite{KP} give a complete result extending~\cite{Re2} for the analytic continuation of roots of hermitian matrices. In our situation it can be reformulated according to the following
\begin{theo}[\cite{KP}]\label{lcone}
  Assume $\lambda_0\notin\sigma_0$. Either dim$\,\cM=1$ then $\lambda_{n_j}$ for $j\in\{1,2\}$ can be relabelled in such a way that they become (real) analytic functions on $V_0$ past the crossing. The same relabelling applies to the associated eigenvectors. Either dim$\,\cM=0$ and then $\ell_0$ is an isolated nodal point which is not a cusp and whose tangent cone lies outside a cone of slope max$\,\alpha/$min$\beta$.
\end{theo}
\begin{proof}
  When $\cM$ is a subset of $ (\cB^{\times})^r$ this is~\cite{KP}, Theorem 6.6. When $\cM\subset\cB^{\times\times}$, upon reducing $V_0$, $\cM$ is the isolated point $\ell_0$ and the tangent space of $\lambda_{n_j}$ at $\ell_0$ is a cone which is a basic property of analytic spaces. More precisely let us give the matrix representation of $T(\ell)$:
\begin{multline}\label{roots}
T(\ell)-\lambda_0 I=\begin{pmatrix}a(\ell) & b(\ell) \\ b(\ell) & c(\ell)\end{pmatrix}=\frac{a(\ell)-c(\ell)}{2}+\begin{pmatrix}d(\ell) & b(\ell) \\ b(\ell) & -d(\ell)\end{pmatrix},\quad d(\ell)=\frac{a(\ell)+c(\ell)}{2}
\end{multline}
where $a,b,c$ are analytic functions satisfying $a(\ell_0)=b(\ell_0)=c(\ell_0)=0$. The eigenvalues of the last matrix are $\pm\sqrt{d^2+b^2}$ whose tangent set at $\ell_0$ is a (non isotropic) cone and $\ell_0$ is not a cusp because the minimal homogeneity degree of the roots is $1$. Finally by Appendix Lemma~\ref{lbband} the function $\lambda_n$ has a gradient bounded by max$\,\alpha/$min$\beta$.

\end{proof}

\begin{rem}
  The tangent cone about a nodal point can't be vertical and flat because otherwise $\lambda_0\in\sigma_0$.
\end{rem}

\begin{defi}\label{def1}
  The first case in the previous theorem will be referred to as \emph{regular crossing}. We denote by $\sigma_1$ the set of points $\ell$ corresponding to the second case of the theorem. If $k^2\notin \sigma_0\cup\sigma_1$ we call $F_{k^2}^0$ a \emph{regular Fermi level}. For such $k$ let us denote by $(\mu_n,v_n)$ the analytically reordered eigenfunctions and eigenvectors.
\end{defi}

Since $\lambda_n$ is defined in $\RR^2$ the function $\mu_n$ is piecewise defined on a subset $\cD$ of $\RR^2$ avoiding the set of preimages of critical points $\sigma_0$ and nodal points $\sigma_1$.

By analytic extension theorem $\mu_n$ extends analytically in a complex neighborhood of $\cD$. Since $\lambda_n$ are piecewise holomorphic in $\CC^2$, $\mu_n$ is thus still piecewise defined in term of $\lambda_n$ in a complex neighborhood of its domain of analyticity.
Contrary to $(\lambda_n,e_n)$ the functions $(\mu_n,v_n)$ are not periodic.

\subsection{Unfolding a regular real Fermi level}

For real $k$ the fiber F$_{k^2}^0$ is a real stratified set and upon taking supp$(\psi)$ small enough, the stratification remains invariant. Generically the fiber is a one dimensional set, with finitely many 1d and 0d strata: 1d strata are connected analytic curves $c_n$ and 0d strata are isolated points corresponding to the crossing of generically two band functions and thus the meeting point of two curves $c_n\cap c_j$. 

Let us now recall the well-known but non written fact that a Fermi level is actually the folding of closed or periodic analytic curves. To see this we make use of the extended Fermi level $F_{k^2}$ which is periodic with $B$ as a unit cell in $\RR^2$. Thus every $c_n\subset F_{k^2}^0$ is repeated by translation of periods of $B$ in $F_{k^2}$.
\begin{defi} 
Let $c$ be a connected curve in $F_{k^2}$. We say that $c$ is {\emph periodic modulo $B$} if $c$ has no boundary or if the vector joining the extremities of $c$ is a linear integer combination of the periods of $B$.
\end{defi}
\begin{lem}
There is a familly $(\tau_n)_n$ of translations by periods of $B$ such that the union over $n$ of the translated arcs $c_n\circ \tau_n$ can be concatenated into closed or periodic (modulo $B$) analytic curves $\Cc_j(k^2)$ of $F_{k^2}$.
\end{lem}

\begin{proof}
See Figure~\ref{fig3}. Let us construct such a set $\Cc(k^2)$ and prove its analyticity. 

Let $c$ be a connected analytical component of $F_{k^2}^0$. It is related to some $\mu_n$. Either it is a closed curve in B and then $c=\Cc(k^2)$ or both end points $(E_1,E_2)$ belong to $\D B$. Since $\mu_n$ is piecewise defined in terms of the $(\lambda_n)_n$ we generically have $E_1\in\lambda_{n_1}^{-1}(k^2)$ and $E_2\in\lambda_{n_2}^{-1}(k^2)$ for some $n_1,n_2$ (which can be equal). Let $E_2'\in \D B$ such that $E_2E_2'$ is a period of $B$. 

Assume that $c$ is the only curve arriving at $E_2$. This means that $\lambda_{n_2}$ is analytic around $E_2$ and by periodicity there is an other curve $c'$ arriving at $E_2'$ such that translating it by $E_2'E_2$ it is an analytic continuation of $c$.

Assume that two curves $c,s$ meet at $E_2$. This means that $s$ is associated to $\lambda_{n_3}$ around $E_2$ and $\lambda_{n_2}$ and $\lambda_{n_3}$ meet at $E_2$. Then by periodicity of $\lambda_{n_2}$ and $\lambda_{n_3}$ there are two curves $c',s'$ arriving at $E_2'$. The choice of the good continuation is by analyticity. Indeed $\mu_n$ continues analytically through $E_2$ thus one of the translated $c',s'$ is associated to $\mu_n$ and thus an analytical continuation of $c$.

Repeating this procedure one gets a sequence of boundary points $E_1,E_2,$ $\ldots, E_j,\ldots$. We claim that the first redundance modulo $B$ must be $E_1$ modulo $B$. Indeed each new connected component which is added to $\Cc(k^2)$ under construction is a translated component of $F_{k^2}^0$. Moreover there is a finite number of such components.
If the first redundance is $E_j$ with $j>1$ then this would mean that there is some $B$-periodic subset $\Cc'(k^2)$ of $\Cc(k^2)$ which does not go through $E_1$ but this is wrong since we can reverse the procedure and see that $E_j$ comes from $E_1$.

\end{proof}

\begin{figure}[!h]
\begin{center}
\includegraphics[width=9cm,height=8cm]{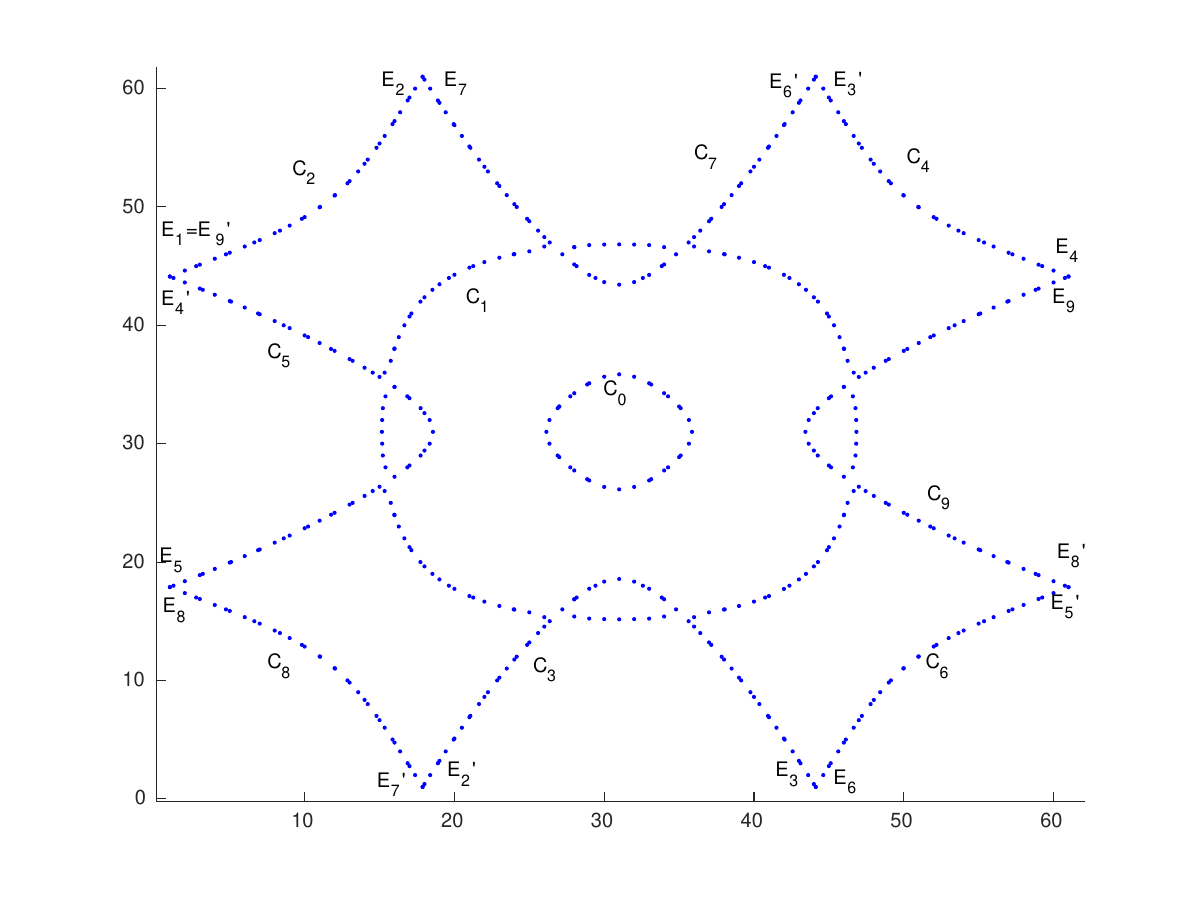}
\includegraphics[width=9cm,height=8cm]{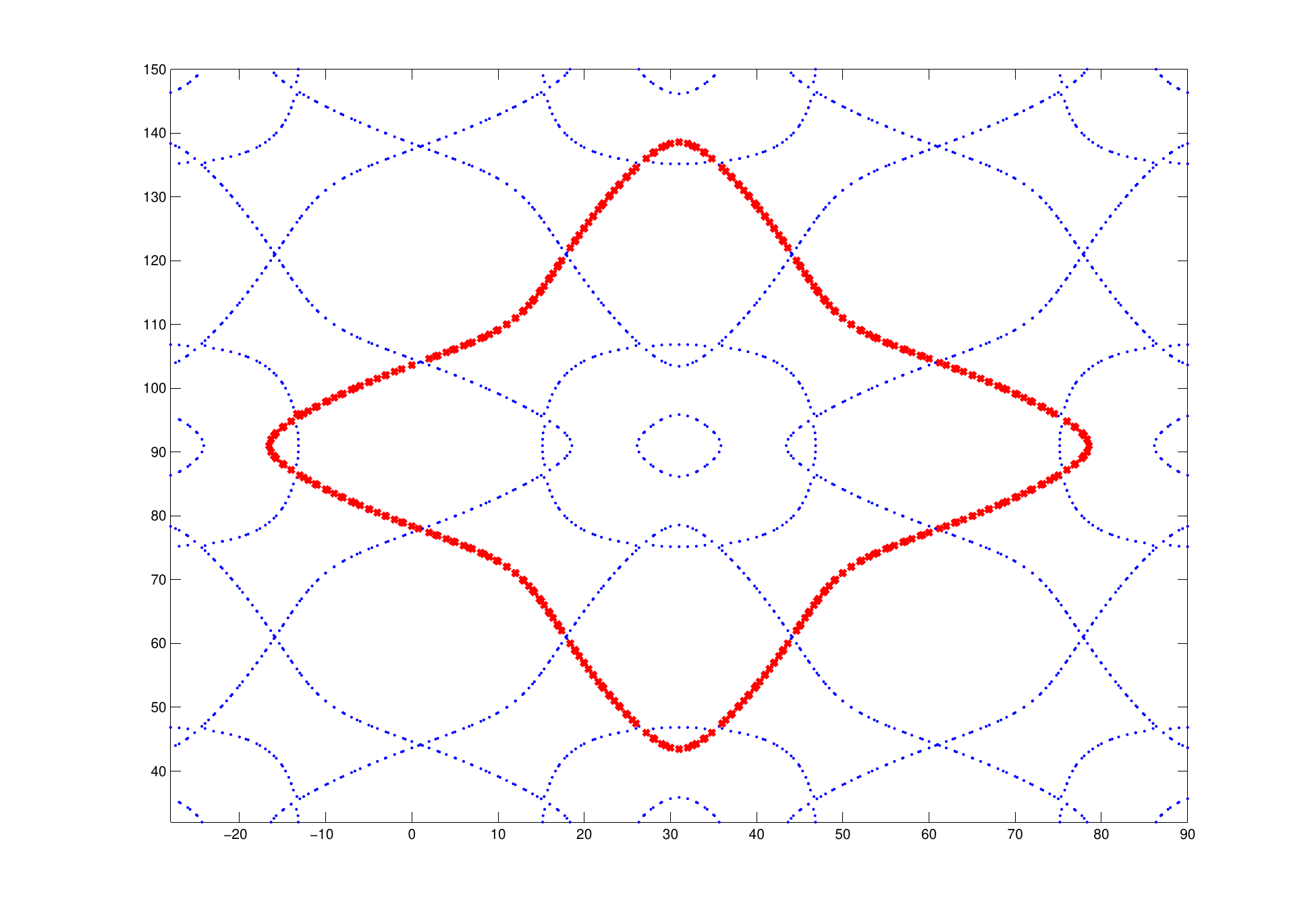}
\caption{Unfolding the Fermi level for $k=1.2$ (left) leads (right) to $\Cc_0(k^2)=c_0$, $\Cc_1(k^2)=c_1$ and $\Cc_2(k^2)=\displaystyle\cup_{2\leq j\leq 9}c_j$ using B-translations (red curve)}\label{fig3}
\end{center}
\end{figure}



\begin{rem}\hspace{1cm}
\begin{itemize}
\item From the proof we see that each $\Cc(k^2)$ is associated to one $\mu_j$ and we denote it by $\Cc_j(k^2)$.
\item On Figure~\ref{fig3} we only have closed curves in an extended Brillouin zone.
\item $\D\Cc_j(k^2)\neq \emptyset$ if there is a unit vector $\hat \imath$ such that the line $t\hat\imath,\ t\in\RR$ does not meet $F_{k^2}^0$. In other words $k^2$ lies in a partial gap of $P$ in the direction $\hat\imath$. In particular this may happen when bands overlap artificially.
\item The curves may cross or may be tangent (see Figure~\ref{fig4}) but this does not make any difference in the subsequent analysis since $k^2\notin\sigma_1$ and thus the crossing of different $\mu_j$ is regular.
  \end{itemize}
\end{rem}
\begin{figure}[!h]
\begin{center}
\includegraphics[width=9cm,height=8cm]{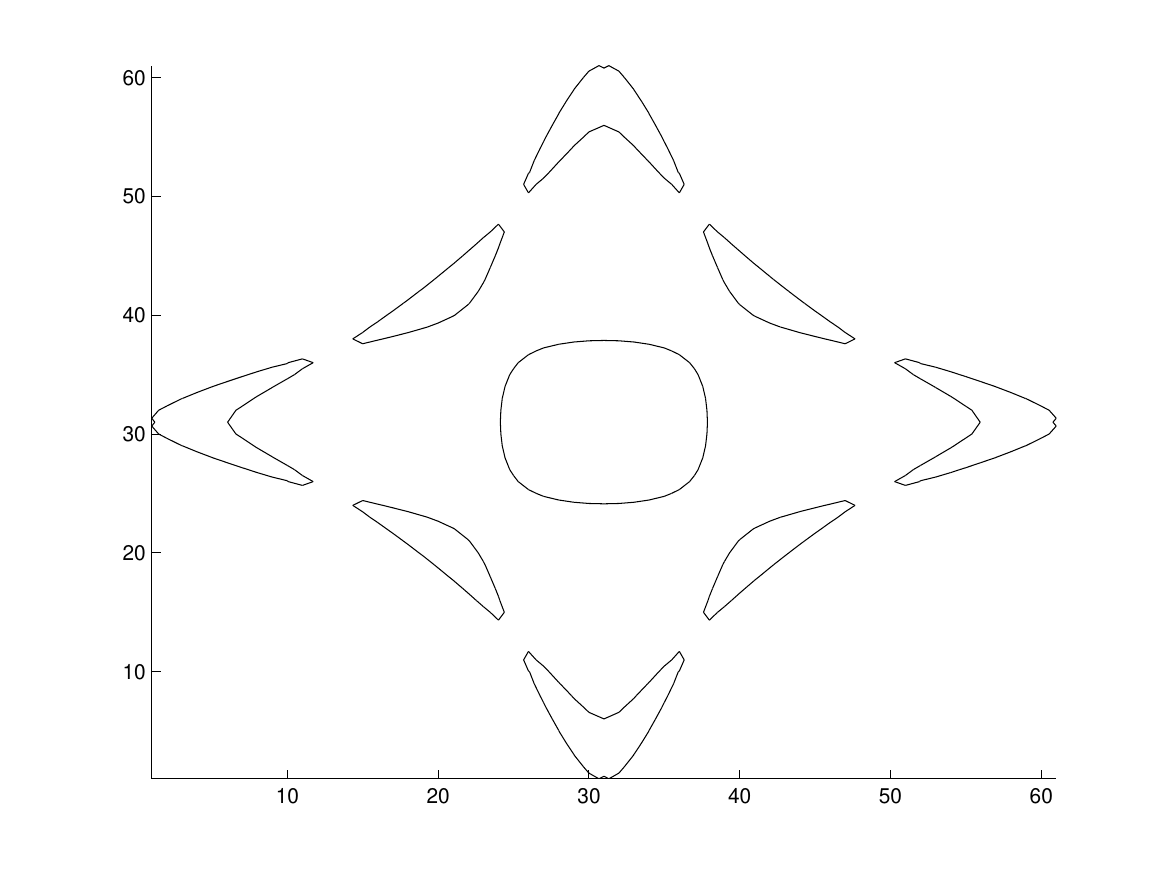}
\includegraphics[width=9cm,height=8cm]{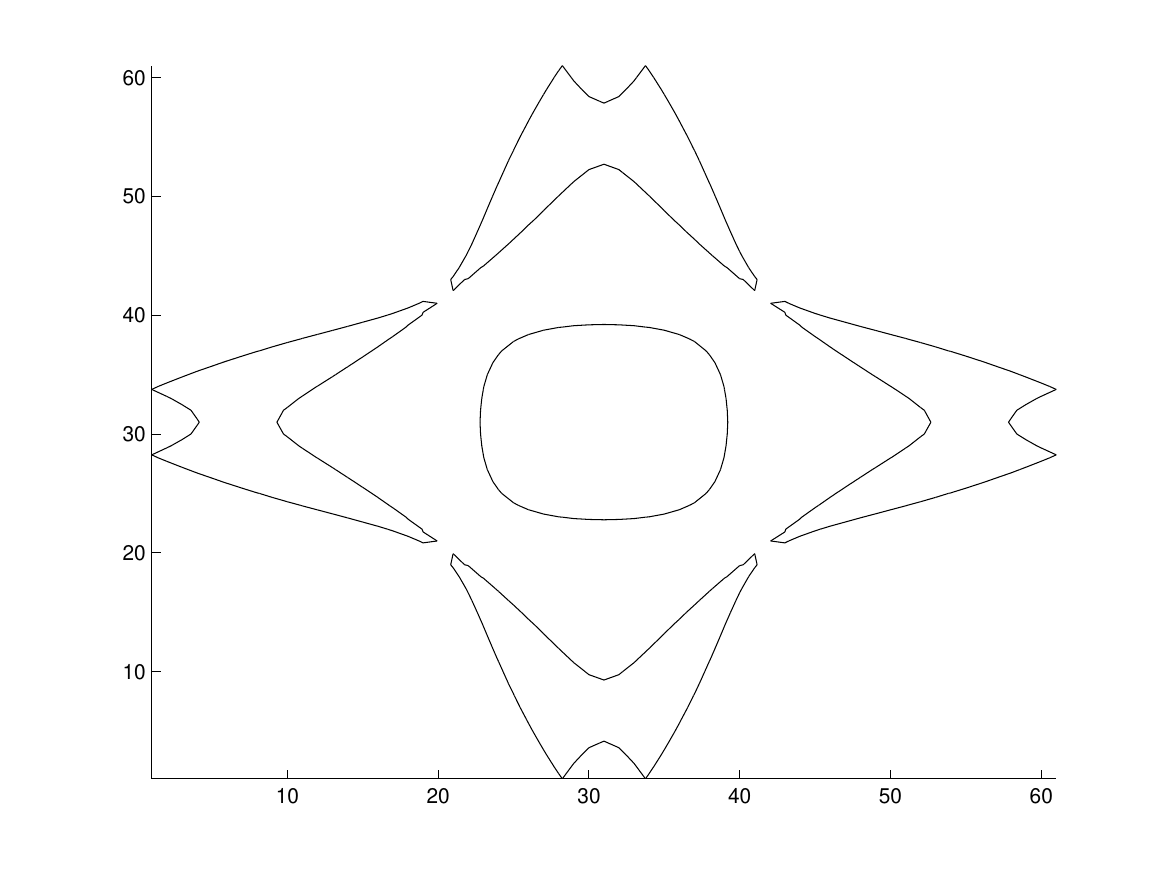}
\caption{Real Fermi level $F_{k^2}^0$ for $k=1.06$ (left) and $k=1.081$ (right, folded)}
\end{center}  
\label{fig4}
\end{figure}

\subsection{Complex extension of a regular real Fermi level}

In view of letting $k$ complex we want to extend the definition of $\Cc_j(k^2)$ to a complex neighborhood of $k^2$.

\begin{lem}\label{lparam}
Let $k\in\RR$ and let $\ell_{j,0}(t)$, $t\in[0,1]$ be a global (real) parameterization of a curve $\Cc_j(k^2)$ in $\RR^2$.
Then for complex $\lambda$ close to $k^2$ one defines $\Cc_j(\lambda)$ through a family of parameterizations $\ell_j(t,\lambda)$, $t\in[0,1]$ such that $\D\ell_j/\D\lambda=\na\mu_{j}/(\na\mu_{j}\cdot\na\mu_{j})$ and $\ell_j(t,k^2)=\ell_{j,0}(t)$ for all $t\in[0,1]$.
\end{lem}

\begin{proof}
  First $\ell_j$ is well-defined because away from $\sigma_0\cup\sigma_1$ the function $\mu_j$ is analytic and $\na\mu_j(k^2)\neq 0$ so one can apply Cauchy-Lipshitz's Theorem with parameter $t$ and show that there exists a complex ball $\mathscr{B}_{k^2}$ independent of $t$ such that $\ell_j(t,\lambda)$ exists for $\lambda\in \mathscr{B}_{k^2}$.
      
Let us show that $\mu_{j}(\ell_j(t,\lambda))=\lambda$ for all $\lambda$ in $\mathscr{B}_{k^2}$. For this we just need to show that the initial parameterization is carried along the flow. We compute and find for all $t\in[0,1]$
$$
\frac{\D}{\D\lambda}\mu_{j}(\ell_j(t,\lambda))=1\quad\text{with}\quad \mu_{j}(\ell_j(t,k^2))=k^2.
$$
Thus $\mu_{j}(\ell_j(t,\lambda))=\lambda$ for any $\lambda$ in a small neighborhood of $k^2$ and $t\in[0,1]$.
\end{proof}


\subsection{Complex displacement about a regular $k^2$-Fermi level}\label{sec22}

\subsubsection{Representation formula}\label{sec221}

Let us resume the guideline presented just before section~\ref{sgeo}.
Let us pick a peculiar $k^2\notin\sigma_0\cup\sigma_1$ and denote by $K$ the support of $\psi$ which we take so small that the topology of $F_{\lambda}^0$ does not change for $\lambda\in K$.
The domain of integration of $u_{1,\eps}$ is $\Psi:=\bigcup_{n\in J_k}\{\ell\in B,\ \psi(\lambda_n(\ell))\neq 0\}$. This set also reads as a union of disjoint Fermi levels: $\Psi\cap\bigcup_{\lambda\in K}F_{\lambda}^0$. From the previous section we have:
$$
\bigcup_{\lambda\in K}F_{\lambda}^0 = \bigcup_j\bigcup_{\lambda\in K} \Cc_j(\lambda)\quad\text{modulo B}
$$
For future use let us denote by
$$
\cA_j:=\bigcup_{\lambda\in K} \Cc_j(\lambda),\quad\text{and}\quad \Fc=\bigcup_{\lambda\in K}F_{\lambda}^0=\bigcup_j\cA_j
$$
Since $\lambda_n$, $e^{i\ell\cdot x} e_n(x,\cdot)$ and $(\hat f(\cdot), e_n(\cdot))_{L^2(W)}$ are $B$-periodic and recalling~\eqref{pw} we get 
\begin{align}\label{eqdeploy}
  u_{1,\eps}&=\sum_{n\in J_{k}}\int_Be^{i\ell\cdot x}\psi(\lambda_n(\ell))\frac{(\hat f,e_n(\ell))_{L^2(W)}}{k_{\eps}^2-\lambda_n(\ell)}e_n(x,\ell)d\ell\nonumber\\
  &=\int_{\Fc}e^{i\ell\cdot x}\psi(\tilde\lambda(\ell))\frac{(\hat f,\tilde e(\ell))_{L^2(W)}}{k_{\eps}^2-\tilde\lambda(\ell)}\tilde e(x,\ell)d\ell, \nonumber\\
  &=\sum_jw_j\quad\text{with}\quad w_j=\int_{\cA_j}e^{i\ell\cdot x}\psi(\tilde\lambda(\ell))\frac{(\hat f,\tilde e(\ell))_{L^2(W)}}{k_{\eps}^2-\tilde \lambda(\ell)}\tilde e(x,\ell)d\ell.
\end{align}

Since $\Cc_j(\lambda)$ is associated to one function $\mu_j$ one has $\tilde e=v_{j}$ (recall definition~\ref{def1}).
Then let us use the explicit parameterization $\ell_j$ of $\Cc_j(\lambda)$ according to Lemma~\ref{lparam} and compute its Jacobian determinant (here $\lambda$ is real). First from $\mu_j(\ell_j(t,\lambda))=\lambda$ we have $\na\mu_j\cdot\D_t\ell_j=0$ and $\na\mu_j\cdot\D_{\lambda}\ell_j=1$ hence
$$
|det(\D_t\ell_j\ \D_{\lambda}\ell_j)|=\frac{|\D_t\ell_j|}{|\na\mu_j|}\na\mu_j\cdot\D_{\lambda}\ell_j=\frac{|\D_t\ell_j|}{|\na\mu_j|}.
$$
Finally
$$
w_j=\int_K\frac{\psi(\lambda)}{k_{\eps}^2-\lambda}\int_{\Cc_j(\lambda)}\frac{e^{i\ell\cdot x}}{|\na\mu_j|}(\hat f(\ell),v_j(\ell))_{L^2(W)}v_j(\ell,x)d\ell d\lambda.
$$

\medbreak

{\bf In what follows we consider any $w_j$ so we drop the index $j$}.

\medbreak



\subsubsection{Light area and shadow}

Before pushing $\lambda$ to the complex plane note that $e^{i\ell(t,\lambda)\cdot x}$ is a quasi-periodic function with respect to $x$. When $\lambda$ is complex then $\ell$ is complex and the sign of
$$
\vp(t,\lambda):=\ell(t,\lambda)\cdot x
$$
can change along $\Cc(\lambda)$ and thus the asymptotic behaviour of $w$ changes drastically. We thus need to caracterize the zeros of this function for $\lambda$ in a small neighborhood of $k^2$.

For complex $\lambda=\tau+iy$ with $y$ close to zero we use Cauchy-Riemann relations to deduce the sign of $\Im\vp$.
Since the parameterization $\ell(t,\cdot)$ is holomorphic and real for $y=0$ the sign of $\Im\vp$ for $y$ going to zero is given by the sign of $\D_y\Im\ell\cdot x$ which is also the sign of $\D_{\tau}\ell\cdot x$ for $y=0$. Now by definition of the parameterization $\ell$ one has $\D_{\tau}\ell\cdot x=(\na\mu\cdot x)/|\na\mu|^2$ for $y=0$ thus the sign of $\Im\vp$ when $y$ is close to zero is given by that of $\na\mu\cdot x$. 

With the homogeneous case in mind we want to reproduce the proof of the asymptotic expansion as in~\cite{Mel}. Thus we split $\Cc(k^2)$ in three parts :
\begin{itemize}
\item a part $\Cc_0$ around the shadow transition (i.e. $\na\mu(\ell(t,k^2))\cdot x\approx0$)
\item and two other parts $\Cc^{\pm}$ such that $\pm\Im\phi>0$ for $\eps>0$ or equivalently such that $\pm\na\mu(\ell(t,k^2))\cdot x>0$.
\end{itemize}
Because the set $\psi=1$ is small we perform this splitting uniformly with respect to $\lambda\in K$.



Let us introduce a partition of unity on $\Cc(k^2)$:
$$
1=\psi_0(t)+\psi_+(t)+\psi_-(t)\quad\text{with}\ \pm\na\mu\cdot x>0\quad\text{on}\quad \text{supp}(\psi_{\pm}).
$$
So $w=w_0+w_+(x)+w_-(x)$ with for $\beta\in\{0,\pm\}$
\begin{equation}\label{3bit}
w_{\beta}=\int_K\frac{\psi(\lambda)}{k_{\eps}^2-\lambda}g_{\beta}(\lambda)d\lambda,\quad g_{\beta}(\lambda)=\int_0^1e^{i\ell\cdot x}\psi_{\beta}(t)\,(\hat f(\ell),v(\ell))_{L^2(W)}v(\ell,x)\frac{|\D_t\ell|}{|\na\mu|}dt.
\end{equation}

\subsubsection{Complex displacement}

We now consider each term $w_{\beta}$ and isolate the leading part.

\medbreak

$\bullet$ For $w_+$ we can choose the integration path in the first integral going above the residue. For this let $I\subset K$ be such that  $\psi=1$ on $I$. Then let us consider $\gamma^+$ a curve homotopic to $I$, encircling
$k_{\eps}^2$ for $\eps$ small enough. See Figure~\ref{fig5}.

\begin{figure}[!h]
\begin{center}
\includegraphics[width=10cm,height=3cm]{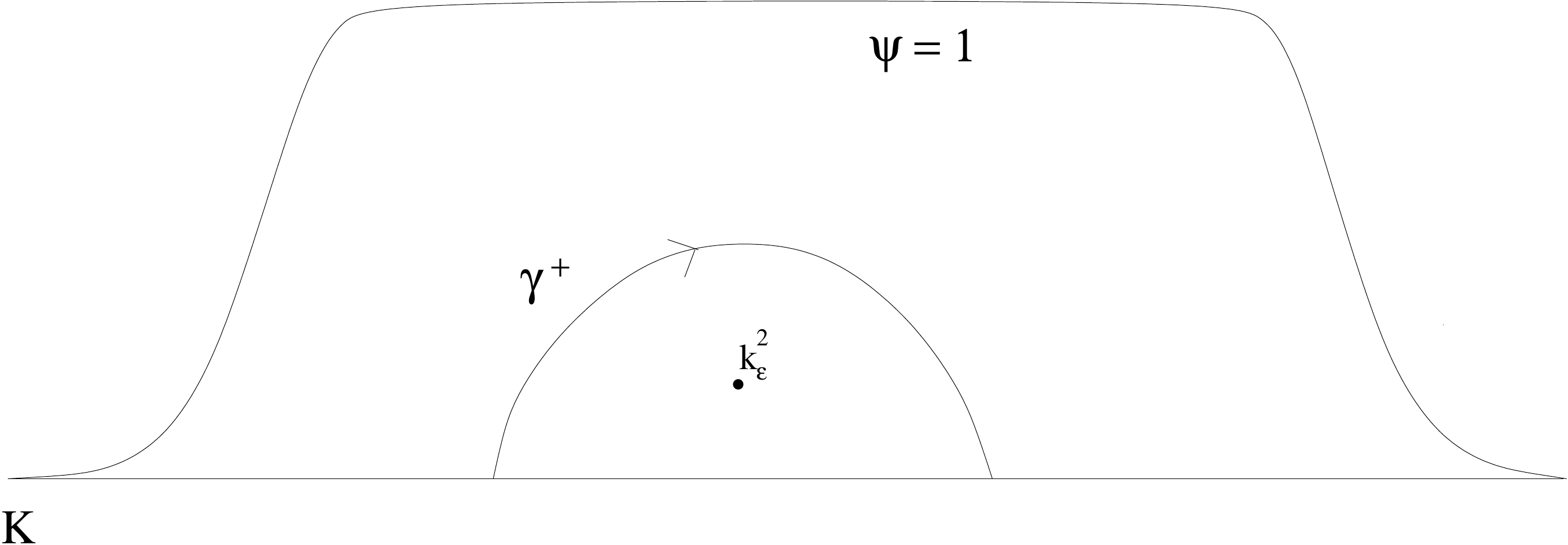}
\caption{The contour $\gamma^+$}
\end{center}  
\label{fig5}
\end{figure}

Then the integrand $g_+$ is holomorphic in the region surrounded by $\gamma^+$ and $I$. Indeed, $\psi=1$ and Lemma~\ref{lparam} provides a holomorphic extension of the jacobian determinant in a neighborhood of $K$ which reads $\pm$det$(\D_t\ell,\D_{\lambda}\ell)$ (with fixed sign on $K$). By the Cauchy-residu formula we get
\begin{equation}\label{Resout}
w_+=2i\pi g_+(k_{\eps}^2)+\int_{K\setminus I}\psi(\lambda)\frac{g_+(\lambda)}{k_{\eps}^2-\lambda}d\lambda+\int_{\gamma^+}\frac{g_+(\lambda)}{k_{\eps}^2-\lambda}d\lambda
\end{equation}
where
$$
g_+(k_{\eps}^2)=\int_{\Cc\cap F_{k^2}^+(x)}\psi_+e^{i\ell\cdot x}(\hat f(\ell),v(\ell))_{L^2(W)}v(\ell,x)\frac{1}{|\na\mu|}\sigma(d\ell).
$$
$\bullet$ Similarly, for $w_-$ we choose a path $\gamma^-$ going below the real axis so that there is no residu:
$$
w_-(x)=\int_{K\setminus I}\psi(\lambda)\frac{g_-(\lambda)}{k_{\eps}^2-\lambda}d\lambda+\int_{\gamma^-}\frac{g_-(\lambda)}{k_{\eps}^2-\lambda}d\lambda.
$$
The previous expressions show that $w_{\pm}$ have limit when $\eps$ goes to zero and $g_+(k^2)$ contributes to formula~\eqref{f1}.
\medbreak

$\bullet$ As for $w_0$ one can take the limit when $\eps$ goes to zero. This limit involves a principal value:
$$
\lim_{\eps\rightarrow 0}\frac{1}{k_{\eps}^2-\lambda}=vp\left(\frac{1}{k_0^2-\lambda}\right)-i\pi\delta_{k_0^2}.
$$ 
The limit thus reads
$$
\lim_{\eps\rightarrow 0} w_0=-i\pi g_0(k_0^2)+vp\int_K\frac{\psi(\lambda)}{k_0^2-\lambda}g_0(\lambda) d\lambda,
$$
where the principal value is bounded.

\section{Boundedness of the residual}

\begin{theo}
The functions $w_0$, $w_+-2i\pi g_+(k_{\eps}^2)$ and $w_-$ belong to $H^1(\RR^2)$ uniformly with respect to $\eps$ (small).
\end{theo}

\begin{proof}
Let us first consider $w_0$ and set 
\begin{equation}\label{notproj}
pf(\ell,x)=(\hat f(\ell),v(\ell))_{L^2(W)}v(\ell,x)\quad\text{and}\quad D_{\ell}:=\pm\det(\D_t\ell,\D_{\lambda}\ell).
\end{equation}
The choice of $\pm$ is such that $D_{\ell}>0$ for real $\lambda$. Formula~\eqref{3bit} reads
$$
w_0 = \int_K\int_0^1 e^{ix\cdot\ell}\psi_0(t)\frac{\psi(\lambda)}{k_{\eps}^2-\lambda} pf(\ell,x)D_{\ell}\,dt d\lambda.
$$
In view of letting $|x|$ go to infinity Let us redefine the phase as $\vp=\ell(t,\lambda)\cdot x/|x|$. It is real and instationnary. Indeed, differentiating the relation $\mu(\ell(t,r))=r$ with respect to $t$ we get $\na\mu\cdot\D_t\ell=0$. Recalling that $\ell$ is defined by $\D_{\lambda}\ell=\na\mu/|\na\mu|^2$ we thus get
\begin{equation}\label{orthogeo}
\D_t\ell\cdot\D_{\lambda}\ell=0.
\end{equation}
Thus the gradient of the phase $\vp$ does not vanish anywhere. Actually on the support of $\psi_0$ we have $\D_t\vp\neq 0$ since $\na\mu(\ell)$ and $x$ are approximately orthogonal thus $\D_t\ell$ and $x$ are approximately collinear. We can thus integrate by parts with respect to $t$.

Let us note that the amplitude is also oscillatory since it is quasi periodic with respect to $\ell$. Integrating by parts requires to differentiate $\hat f$ with respect to $\ell$ and since $\D_{\ell}\hat f=-i\wh{xf}$ we need to take $f$ in weighted $L^2$ space or even in the Schwarz space if one wants to get a full series expansion of $w_0$ in inverse powers of $|x|$.
  
In order to prove that $w_0\in L^2(\RR^2)$ we need to integrate twice by parts because $|x|^{-1}\notin L^2(\RR^2\setminus B(0,1))$. Because $\psi_0$ has compact support we get
$$
w_0=\int_{\RR^+}\frac{\psi(r)}{k_{\eps}^2-r}\int_0^1\frac{-1}{|x|^2}e^{i|x|\vp}\left(\D_t\frac{1}{\D_t\vp}\right)^2\left(\underbrace{\psi_0(t) pf(\ell,x)D_{\ell}}_{h}\right) dt dr.
$$
The second derivative $(\D_t\frac{1}{\D_t\vp})^2h$ expresses as a sum of terms of the form
$$
\frac{G(t)(\D_t^a\vp)^b\D_t^c pf\D_t^dD_{\ell}}{(\D_t\vp)^e},\quad (a,b,c,d,e)\in\NN^5\quad b+c+d\leq 2,\quad a\leq 3,\ e\leq 4, 
$$
where $G$ is bounded. Since $\ell$ is analytic with respect to $t$, $\D_t^a\vp$ and $\D_t^dD_{\ell}$ are bounded and $\D_t\vp$ is lower bounded as we explained above. The only term left is $\D_t^c pf$ for which one needs to estimate $\D_{\ell}^cpf$. More precisely, there is a constant $\kappa$ such that
$$
|x|^2|w_0|\leq \kappa\|pf(\cdot,x)\|_{W^{2,1}(\cD)},\quad\text{where}\quad \cD=\bigcup_{\lambda\in K}\Cc(\lambda).
$$
\begin{lem}
There is a periodic function $q>0,\ q_{|_W}\in L^2(W)$ such that for any $k\in\NN$
$$
\|\D_{\ell}^k pf(x,\cdot)\|_{L^1(\cD)}\leq q(x)\|\langle x\rangle^k f\|_{L^2(\RR^2)},\qquad \langle x\rangle=\sqrt{1+|x|^2}.
$$
\end{lem}
See the proof of the lemma~\ref{lfin} below dealing with a complex extension of this result.

From the previous estimate we finally get the existence of a positive constant $c$ such that
$$
\|w_0\|_{L^2(\RR^2)}\leq c\|\langle x\rangle^2f\|_{L^2(\RR^2)}.
$$
Next, estimating $\|\na w_0\|_2$ is equivalent to estimating $\|\sqrt{\alpha}\na w_0\|_2$ but
$$
\|\sqrt{\alpha}\na w_0\|_2^2=(Pw_0,w_0)_2\leq\max_{\ell\in B}\mu(\ell)\|w_0\|_2^2
$$
since $P$ commutes with the integrals and $Pv = \mu v$. 

\bigbreak

Let us now turn to $\tilde w_+=w_+-2i\pi g_+$ and $w_-$. So we consider
$$
\tilde w_+=\int_{\gamma^+\cup K\setminus I}\int_0^1e^{ix\cdot\ell}\psi_+(t)\frac{\psi(\Re\lambda)}{k_{\eps}^2-\lambda} pf(\ell(t,\lambda),x)D_{\ell}dt d\lambda.
$$
The main change is due to complex values and the fact that on the support of $\psi_+$ the phase $\vp=\ell\cdot x/|x|$ is not instationnary with respect to $t$ but it is instationnary with respect to $\lambda$. Indeed the real part of the phase for $\lambda\in\gamma^+$ is a small perturbation of its values for $\lambda\in I$ and for such real $\lambda$ the vector $\D_{\lambda}\ell$ is collinear to $\na\mu$ which is not orthogonal to $x$.
 
Let us remark that when the phase takes negative real values these are of order of the imaginary part of $\gamma^+$ which is small since we need to remain in the domain of analyticity of $\mu$. So we do not use the exponential decrease but the nonstationarity of the phase. To integrate by parts we just need to provide enough regularity by choosing the contour $\gamma^+\cup(K\setminus I)$ smooth enough. Integrating twice by parts with respect to $\lambda$ we get
$$
\tilde w_+= -\int_{\gamma^+\cup(K\setminus I)}\int_0^1\frac{1}{|x|^2}e^{-i|x|\vp}\psi_+(t)\left(\D_{\lambda}\frac{1}{\D_{\lambda}\vp}\right)^2\frac{\psi(\lambda)}{k_{\eps}^2-\lambda} pf(\ell(t,\lambda),x)D_{\ell}dt d\lambda
$$
Let $\lambda(s),\ s\in I$ be a parameterization of the curve $\gamma^+\cup(K\setminus I)$ and let $\phi(t,s)=\ell(t,\lambda(s))$. Also let $\Cc^+(\lambda)=\Cc(\lambda)\cap\{\psi^+\neq 0\}$. Then $\phi$ is a diffeomorphism from $\Delta:=[0,1]\times K$ to $\bigcup_{\lambda\in\gamma^+}\Cc^+(\lambda)$ because $\Im\lambda$ is small and for $\lambda=Id$ it is so by~\eqref{orthogeo}. Then we can estimate $\tilde w_+$ by
$$
|x|^2|\tilde w_+(x)|\leq \frac{c}{\alpha^3}\| pf(\phi,\cdot)\|_{W^{2,1}(\Delta)},\quad \alpha=\text{dist}(\gamma^+,k_{\eps}^2).
$$
\begin{lem}\label{lfin}
There is a periodic function $q>0,\ q\in L^2(W)$ and a constant $c>0$ such that
$$
\|\D_{\ell}^k pf(x,\phi)\|_{L^1(\Delta)}\leq q(x)\|<\cdot>^{k+1} e^{\max|\Im\phi||\cdot|}f\|_{L^2(\RR^2)}.
$$
\end{lem}
\begin{proof}
Let us show the lemma for $k=0$ (for bigger $k$ the proof is similar because $\D_{\ell}\hat h=-i\widehat{xh}$). $e$ being holomorphic and periodic on $\Delta$ it is bounded in a small complex neighborhood of the Brillouin zone. Hence $v$ is bounded too and $pf$ can be estimated by
$$
\| pf(x,\phi)\|_{L^1(\Delta)}\leq \|v(x,\phi)\|_{L^2(\Delta)}\|v(\cdot,\phi)\|_{L^{\infty}(\Delta;L^2(W))}\|\hat f(\phi)\|_{L^2(W\times\Delta)}.
$$
Since $v$ is holomorphic with respect to $\ell$ it is bounded in $\Delta$ so that the first term is a $L^2(W)$ periodic function of $x$. For the second term notice that $\|e(\ell)\|_{L^2(W)}$ (hence $\|v(\ell)\|_{L^2(W)}$) is identically equal to $1$ on the real line and continuous by Lebesgue continuity theorem.

Finally we give a crude estimate of $\|\hat f(\phi)\|_{L^2(W\times\Delta)}$ as follows. First
$$
|\hat f(\phi)|\leq\sum_{n\in\ZZ^2}|f(x+2\pi n)|e^{\max|\Im\phi||x+2\pi n|}=\wh{|f|h}(0)\quad\text{where}\quad h(x)=e^{\max|\Im\phi||x|}.
$$
Then setting $g=|f|h$ one has 
$$
  \|\hat g(0)\|_{L^2(W)}^2=\sum_{n,m}\int_Wg(x+2\pi n)g(x+2\pi m)\leq \left(\sum_n\|g\|_{L^2(W+2\pi n)}\right)^2
$$
and
$$
\sum_n\|g\|_{L^2(W+2\pi n)}\leq c\sum_n\frac{\|\langle\cdot\rangle g\|_{L^2(W+2\pi n)}}{<n>}.
$$
Now applying Cauchy-Schwarz inequality the last term is bounded by $c\pi^2/6\|\langle\cdot\rangle g\|_{L^2(\RR^2)}^2$. Finally $\|\hat f(\phi)\|_{L^2(W\times\Delta)}\leq\mu(I)c\pi^2/6\|\langle\cdot\rangle g\|_{L^2(\RR^2)}^2$ where $\mu(I)$ is the measure of $\Delta$.

\end{proof}
From the previous estimate and since $\phi$ is a diffeomorphism one has $\max_D|\Im\phi|\leq c\alpha$ hence
$$
\|\tilde w_+\|_{L^2(\RR^2)}\leq \frac{c}{\alpha^3}\|\langle x\rangle^3e^{c\alpha |\cdot|}f\|_{L^2(\RR^2)}.
$$
This estimate shows that one must take $f$ exponentially decaying for $\tilde w_+$ to belong to $L^2(\RR^2)$. The same estimate holds for $w_-$.
\end{proof}




\section{Asymptotic behavior}

Let us now turn to the far field asymptotics of the residus $g_+(k^2)$. The latter is an oscillatory integral whose phase is stationary on $\Cc^+(k^2)=\Cc(k^2)\cap\{\psi_+\neq 0\}$ when $\D_t\vp=\D_t\ell\cdot x/|x|=0$. Since $\vp$ is analytic as a function of $t\in[0,1]$ it has finitely many extrema. Let $t_*$ and $\ell_*=\ell(t_*)$ be a critical point (resp. value). Since $\vp$ depends on $\theta=arg(x)$, $t_*$ is a function of $\theta$.
For comparison purpose recall that in the homogeneous case (i.e. $\alpha$ and $\beta$ constant) there is just one curve $\Cc$ which is a circle. There is one outgoing stationary point $t=\theta$ defined on $[0,2\pi[$.
When $\Cc$ is not convex or not closed then $t_*$ is only defined on a subintervall $I_*$ of $[0,2\pi[$. A point $\ell_*$ on a convex part of $\Cc$ moves anti-clockwise as $\theta$ increases while points on concave parts move clockwise. The extremities of $I_*$ correspond to inflexion points of $\Cc$ and there $dt_*/d\theta$ is infinite and two critical values merge or emerge.

The derivative $\D_t^2\vp$ is related to the curvature of $\Cc$. Indeed, since $x$ and $\D_t\ell$ are orthogonal we have
\begin{equation}\label{curv}
\D_t^2\vp = \D_t^2\ell\times\frac{\D_t\ell}{|\D_t\ell|}=\kappa|\D_t\ell|^2
\end{equation}
where $\kappa$ is the curvature.
When the curvature vanishes the phase degenerates. This is a well known situation in optics: if the phase is first order degenerated then it is cubic and the integral around this inflexion point is a Airy function~\cite{Ai}. Let $\theta_*$ be such an inflexion point and $\chi$ a test function supported about $t_*$ such that $\chi(t_*)=1$.
Let us split $g_+=g_{+*}+\tilde g_+$ where $g_{+*}$ is defined like $g_+$ but replacing $pf=(\hat f,v)v$ by $\chi pf$. From~\cite{Hor} Theorem 7.7.18 there are functions $\alpha$, $\gamma$ and $w_*$ and $\tilde w_*$ such that $r\alpha(\theta_*)=x\cdot\ell_*$, $\gamma(\theta_*)=0$ and for $\theta$ in a neighborhood of $\theta_*$ (orthogonal direction to $\D_t\ell(t_*$) the following asymptotics holds
\begin{equation}\label{Airy}
g_{+*}(x) = \frac{e^{ir\alpha(\theta)}}{r^{1/3}}\left (Ai(\gamma(\theta) r^{2/3}) w_*(x)+ \frac{1}{r^{1/3}}Ai'(\gamma(\theta) r^{1/3}) \tilde w_*(x)\right ) + O\left(\frac{1}{r^{3/2}}\right)
\end{equation}
where $O(r^s)$ is with respect to $L^{\infty}$ norm.
Actually~\cite{Hor} is true for $pf$ not depending on $x$. However the formula still holds true because $x$ can be considered as a parameter and $pf$ is a bounded (periodic) function of $x$. 
From Appendix~\eqref{Ai1} and~\eqref{Ai2} we find that $\alpha(\theta)=\ell_*x/|x|$ and $w_*=d_*v$ where $d_*$ is a non zero coefficient.
Moreover there is a full Taylor expansion in powers $1/r^{1/3+n}$ and $1/r^{2/3+n}$ and the exponent in the $O()$ term is one order bigger than the fisrt term. As for the first term its order is $O(r^{-1/2})$ because the oscillatory part of $Ai$ and $Ai'$ decays like $r^{-1/4}$. Note that non degenerate critical points also give $O(r^{-1/2})$ amplitude. Only the "far field" pattern is different since $Ai$ is exponentially decaying for positive arguments.


To get formula~\eqref{exp} of theorem~\ref{t2} we resume the index $j$ (of the curve $\Cc_j$) which we dropped in section~\ref{sec22}. Take $k$ as in Theorem~\ref{t1} and denote by $\ud\ell_{jp}$ the finitely many inflexion points of the curve $\Cc_j$.
Denote by $t_{jn}$ the (finitely many) critical points of $\vp_j$ and $\ell_{jn}=\ell_j(t_{jn})\in\Cc_j^+(k^2)$ the related Floquet numbers. These are analytic functions of $\theta\in I_{jn}$ where $I_{jn}$ is an open intervall whose extremities are the angles $\theta_{jp}$ associated to the $\ud\ell_{jp}$.
Then for $\theta$ away from the $\theta_{jp}$ the following asymptotics holds
\begin{equation}\label{Staphas}
g_{j+}(x,k^2)=i\sqrt{2\pi} e^{-i\pi/4}\sum_n 1_{I_{jn}}(\theta)e^{i\ell_{jn}\cdot x}\frac{(\hat f(\ell_{jn}),v_j(\ell_{jn}))_{L^2(W)}}{(\D_t^2\vp_j(t_{jn},k^2)|x|)^{1/2}}v_j(\ell_{jn},x)\frac{|\D_t\ell_j(t_{jn})|}{|\na\mu_j(\ell_{jn})|}+R_2
\end{equation}
As for $\theta$ about $\theta_{jp}$ let $n_1$ and $n_2$ be the index of the intervals $I_{jn}$ whose mutual end is $\theta_{jp}$. Then splitting $g_{j+}=g_{j+*}+\tilde g_{j+}$ as above we find
\begin{multline}
g_{j+}(x,k^2)=i\sqrt{2\pi} e^{-i\pi/4}\sum_{n\notin\{n_1,n_2\}} 1_{I_{jn}}(\theta)e^{i\ell_{jn}\cdot x}\frac{(\hat f(\ell_{jn}),v_j(\ell_{jn}))_{L^2(W)}}{(\D_t^2\vp_j(t_{jn},k^2)|x|)^{1/2}}v_j(\ell_{jn},x)\frac{|\D_t\ell_j(t_{jn})|}{|\na\mu_j(\ell_{jn})|}\\
 + \sum_p \frac{e^{ix\cdot\ud\ell_{jp}}}{r^{1/3}}\left (\tilde c_j Ai(\gamma_{jp}(\theta-\theta_{jp})r^{2/3})v_j(x,\ud\ell_{jp})+\frac{1}{r^{1/3}}Ai'(\gamma_{jp}(\theta-\theta_{jp})r^{2/3})\tilde w_{jp}(x)\right)+ R_2
\end{multline}
In both cases $R_2=O(r^{-3/2})$. In particular it belongs to $H^1(\RR^2)$.
Using~\eqref{curv} and recalling~\eqref{Resout},\eqref{eqdeploy},\eqref{split1} and keeping one index on a bigger range we arrive at~\eqref{exp},\eqref{fAi}.

\section{Uniqueness} 


Let $u$ be an outgoing solution of~\eqref{e1} according to Definition~\ref{def1} with $f=0$ (no source term).
Assumption~\ref{ass2} about the regularity of the coefficients entails that $\nabla u$ is continuous except across (the smooth) discontinuity of $\alpha$. Then the same holds for $R$ since the leading terms (of $u$) are continuous. This allows to integrate $\na R$ on a curve. Then uniqueness will follow from
\begin{lem}\label{l3}
As $t$ goes to infinity
\begin{equation}
 \Im (Pu,u)_{L^2(C_t)}=\sum_k\int_{S_k}\frac{|c_k(\theta)|^2}{(\beta \tilde e(\ell_k(\theta)),\tilde e(\ell_k(\theta)))_2}\na\tilde\lambda(\ell_k(\theta))\cdot\vec n\,d\sigma+\sum_{k'}\alpha_{k'}|\tilde c_{k'}|^2+o(1)
\end{equation}
where $C_t$ is the circle of radius $t$ and $S_k$ the part of the unit circle related to $I_k$ and $\tilde e,\tilde\lambda$ are defined in~\eqref{pw} and $\alpha_{k'}>0$.
\end{lem}

\begin{proof}
Integrating by part in the disk $D_t$ of radius $t$ and taking the imaginary part yields
$$
0=\Im (Pu,u)_{L^2(D_t)}=\Im (\alpha\D_n u,u)_{L^2(C_t)}.
$$
Let us examine this last expression expanding $u$ according to~\ref{def1}(\eqref{convexp},\eqref{inflexp}). First to avoid explicit bounds in evaluating integrals let us introduce a partition of unity of $[0,2\pi[$: 
$$\sum_k\chi_k+\sum_{k'}\tilde\chi_{k'}=1$$
where the union of the supports of the $\chi_k$ is $[0,2\pi[\setminus\cN$ and the support of $\tilde\chi_{k'}$ is $\cN_k'$ with $\cN_k'$ slightly bigger than $\cN_k$ centered at $\theta_k$ and mutually disjoint. Put
\begin{itemize}
\item $u_k(x)=\ds\frac{1}{\sqrt{r}}\chi_k(\theta)c_k(\theta)e^{i\ell_k(\theta)\cdot x}v_k(x,\ell_k(\theta))$
\item For $\theta\in\cN_k'$ set
  $\tilde u_k(x)=\ds \tilde c_k\tilde\chi_k\frac{e^{ix\cdot\ud\ell_k}}{r^{1/3}} \Big ( Ai( r^{2/3}\gamma_k(\theta-\theta_k))v_k(x,\ud\ell_k)+\frac{1}{r^{1/3}}Ai'( r^{2/3}\gamma_k(\theta-\theta_k))\tilde w_k(x,\ud\ell_k)\Big)$
\end{itemize}
With these notations~\eqref{convexp} and~\eqref{inflexp} read
$$
u(x)=\sum_k u_k+\sum_{k'}\tilde u_{k'}+R.
$$
Let us break $(\alpha\D_n u,u)_{L^2(C_t)}$ accordingly :
$$
\begin{array}{ll}
1.\ T_0(t)=(\alpha \D_n R,R)_{L^2(C_t)} & 2.\ T_k(t)=(\alpha \D_n R,u_k)_{L^2(C_t)}\ \text{or}\ (\alpha \D_n u_k,R)_{L^2(C_t)}\\
3.\ T_{ij}=(\alpha \D_n u_i,u_j)_{L^2(C_t)}\ \text{where}\ I_i\cap I_j\neq\emptyset & 4.\ \tilde T_k=(\alpha \D_n R,\tilde u_k)_{L^2(C_t)}\ \text{or}\ (\alpha \D_n \tilde u_k,R)_{L^2(C_t)} \\
5.\ \tilde T_{ij}=(\alpha \D_n u_i,\tilde u_j)_{L^2(C_t)}\ \text{or}\ (\alpha \D_n \tilde u_i, u_j)_{L^2(C_t)} & 6.\ T_i^{\sharp}=(\alpha \D_n \tilde u_i,\tilde u_i)_{L^2(C_t)}
\end{array}
$$

{\bf 1.} Let us first consider the term $T_0(t)$. Since $R\in H^1$ and is piecewise continuous the function $T_0$ is integrable and continuous on $\RR^+$ so $\lim_{t\rightarrow \infty}f(t)=0$.


{\bf 2.} Then we have $T_k(t)=o(1)$ because $u_k\in L^2(C_t)$ uniformly with respect to $t$ and $\|\na R\|_{L^2(C_t)}$ is continuous and belongs to $L^2(\RR^+)$. Similarly $(\alpha \D_n u_k,R)_{L^2(C_t)}=o(1)$.


{\bf 3.} First compute the gradient of $u_k$ with respect to $x$. Since $u_k$ is given by a profile depending on $r,\theta,x$ let us set $u_k(x)=U_k(r,\theta,x)$ and let us use the chain rule
$$
\na u_k(x)=\D_r U_k(r,\theta,x)\vec e_r+\frac{1}{r}\D_{\theta}U_k\vec e_{\theta}+\na_x U_k
$$
$$
\na u_k(x)=\frac{e^{i\ell_k\cdot x}}{\sqrt{r}}\left (-\frac{c_k(\theta)}{r}\chi_kv_k\vec{e}_r + c_k(\theta)\chi_k(\na_x+i\ell_k)v_k+\D_{\theta}(c_k\chi_kv_k)\frac{1}{r}\vec e_{\theta}\right ).
$$
Observe that thanks to the condition of stationary phase $\D_{\theta}e^{i\ell_k(\theta)\cdot x}=0$.
Hence
\begin{equation*}
  T_{ij}=\int_{C_t}\chi_i(\theta)\chi_j(\theta)\frac{\alpha c_i\bar c_j}{t}e^{ix\cdot(\ell_i-\ell_j)}((\na_x+i\ell_i(\theta)) v_i\cdot\vec{n})\bar v_jd\sigma+O(t^{-1}),\quad\text{where}\ \theta=arg(\sigma)
\end{equation*}
For $j\neq i$ we use the stationary phase theorem as we did in the proof of theorem 3.1 to show that the integral is of lower order. Indeed using polar coordinates the phase $\phi(\theta)=t\vec e_r(\theta)\cdot(\ell_i-\ell_j)$ has derivative
$$
t^{-1}\phi'(\theta) = \vec e_{\theta}(\ell_i-\ell_j) + \vec e_r(\ell_i'-\ell_j').
$$
where by definition $\ell_k'(\theta)\cdot x=0$ for all $k$. Then assuming as in Theorem~\ref{t2} that the phase is a Morse function it has at most a finite number of stationary points. Hence the integral is $O(t^{-1/2})$.

For $i=j$ let us set $p_j(x,\theta)=\alpha(x)\vec{n}(\theta)\cdot(\na_x+i\ell_j(\theta)) v_j(x,\ell_j(\theta))\bar v_j(x,\ell_j(\theta))$ which we consider for any $\theta$ and $x$ (not only for $\theta=x/|x|$). This is a periodic function with respect to $x$ which belongs to $L^{\infty}(W\times I_j)$ thanks to the assumption on $\alpha$. Rescaling by $t$ we get 
$$
T_{jj} = \int_{S_j}\chi_j(\theta)|c_j(\theta)|^2p_j(tx,\theta)d\sigma(x)+O(t^{-1}).
$$
From~\cite{BLP} p.94 when $t$ goes to infinity the latter converges to
$$
\int_{I_j}\chi_j(\theta)|c_j(\theta)|^2p_{j0}(\theta)d\theta.
$$
Indeed one can uniformly approximate $p_j$ by smooth trigonometric polynomial for which the limit holds using stationary phase theorem.
Using identity~\eqref{speed} in Appendix and definition~\ref{pw} gives the formula of the lemma.\\



{\bf 4.} First $\tilde u_k$ is a sum of two terms
$$
\tilde u_k=\tilde u_{k1}+\tilde u_{k2}
$$
where the second term is decaying quiker than the first for large $r$. So it is enough to deal with the first. To show that $\tilde T_k=o(1)$ we proceed as for $T_k$ showing that $\tilde u_{k1}$ belongs to $L^2(C_t)$ uniformly with respect to $t$:
$$
\|\tilde u_{k}\|^2_{L^2(C_t)}\leq |\tilde c_k|^2\|v_k(\ud\ell_k)\|_{\infty}t^{1/3}\int_0^{2\pi}\tilde\chi_k^2(\theta)Ai^2(\gamma_k(\theta) t^{2/3}) d\theta.
$$
Since $Ai(-\rho)\underset{+\infty}{\sim}|\rho|^{-1/4}$ the last integral is of order $t^{-1/3}$. Hence $\tilde T_k=o(1)$.

{\bf 5.} 
Hiding $\tilde u_{k2}$ in a $O(t^{-1/3})$ term we get
$$
\tilde T_{ik}=\tilde c_k\int_{C_t}\frac{\alpha(x) c_i}{t^{5/6}}\chi_i(\theta)\chi_k(\theta)e^{ix\cdot(\ell_i-\ud\ell_k)}((\na_x+i\ell_i(\theta)) v_i(x,\ell_i(\theta))\cdot\vec{n})Ai(\gamma_k(\theta-\theta_k)t^{2/3})\bar v_k(x,\ud\ell_k)d\sigma+O(t^{-1/3})
$$
When the phase is stationnary at $\theta_k$ then the leading term has amplitude $t^{1-5/6-1/2}Ai(\gamma_k(0)t^{2/3})=O(t^{-1/2})$. Otherwise it is $O(t^{-1})$.

{\bf 6.} As in step 3 we first compute
$$
\na_x\tilde u_{k1}=\tilde c_k\tilde\chi_k(\theta)\frac{e^{ix\cdot\ell_k}}{r^{1/3}}Ai(\gamma_k(\theta-\theta_k) r^{2/3})(\na+i\ud\ell_k)v_k(x,\ud\ell_k)+O(r^{-3/2}).
$$
shifting by $\theta_k$ and setting $p_k=\alpha(x)\vec{e}_r\cdot(\na_x+i\ell_k)w_k\bar w_k$ yields
$$
(\alpha\D_n \tilde u_{k1},\tilde u_{k1})_{L^2(C_t)}=|c_k|^2t^{1/3}\int_0^{2\pi}\tilde\chi_k^2(\theta)Ai^2(t^{2/3}\gamma_k(\theta))p_k(x,\theta)d\theta+O(t^{-1}).
$$
Using again the periodicity of $p_k$ with respect to $x$ as in {\bf 3.} we can adapt the stationary phase method to each frequency of $p_k$ giving an integral of size $O(t^{-1/3})$ if the phase is non stationary and $O(t^{-1/4})$ if it is stationary. So only the mean value of $p_k$ contributes at infinity:
$$
(\alpha\D_n \tilde u_{k1},\tilde u_{k1})_{L^2(C_t)}=|\tilde c_k|^2(\beta\tilde e(\ud\ell_k),\tilde e(\ud\ell_k))\int_0^{2\pi}\tilde\chi_k^2(\theta)t^{1/3}Ai^2(t^{2/3}\gamma_k(\theta))\na\tilde\lambda(\ud\ell_k)\cdot\vec e_r(\theta) d\theta+o(1)
$$
When $t$ goes to infinite the last integral converges to $d_k\na\tilde\lambda(\ud\ell_k)\cdot\vec w$ which is positive because $d_k>0$ and $\vec w$ is close to $\vec e_r(\theta_k)$ which is collinear and directed like $\na\tilde\lambda(\ud\ell_k)$.
\end{proof}

\begin{proof}[Proof of theorem~\ref{t3}]
Take two outgoing solutions of~\eqref{e1} and denote by $w$ their difference. It is an outgoing solution of $Pw=0$. Let us still denote by $c_k$ and $\tilde c_{k'}$ its asymptotic coefficients (according to~\eqref{convexp} and~\eqref{inflexp}). Since $\na\tilde\lambda(\ell_k)\cdot n>0$ in $S_k$ the previous lemma implies that $c_k(\theta)=0$ almost everywhere and $\tilde c_{k'}=0$. Thus $w\in H^1(\RR^2)$ but since the spectrum of $P$ is purely essential from Assumption~\ref{ass1} it follows that $w=0$.
\end{proof}

\section{Singular Fermi levels}

\subsection{Nodal point}\label{secnodal}

Let us denote by $\lambda_0=k_0^2$ a point of $\sigma_1$. We claim that the resolvent $R(z):=(P-z)^{-1}: L^2_{comp}\rightarrow L^2_{loc}$ about such a point is continuous but not holomorphic, a fact that was not mentioned by~\cite{G}. Indeed let us consider~\eqref{roots} again which is the case of two bands meeting non critically at a single point $\ell_0$ which I assume to be $0$. 

Then the band functions are of the form 
$$
\lambda_{\pm}(\ell)=g(\ell)\pm\sqrt{a^2(\ell)+b^2(\ell)}
$$
where $g,a,b$ are analytic, $a$ or $b$ do not vanish identically and the three functions vanish at $\ell_0$. Note that if we consider the touching bands together this gives an arc analytic function~\cite{KP}, Theorem 7.2.


Next since $\lambda_{\pm}$ are non critical at $\ell_0$ it implies that the minimal degree of Taylor's expansion of $a^2+b^2$ has to be less or equal to two thus $a$ or $b$ has to be linear at first order. We thus consider the following functions
$$
\lambda_{\pm}(\ell)=\lambda_0\pm\sqrt{\ell_1^2+\ell_2^{2m}},\quad m\in\NN^*.
$$

With $z\in\CC$ close to $\lambda_0$ the leading term of the resolvent $u_1(x,z)$ (it corresponds to $u_{1,\eps}$ when $z=\lambda_0+i\eps$) reads
\begin{align*}
 u_1(x,z)&=\int_Be^{i\ell\cdot x}\left(\psi(\lambda_-(\ell))\frac{pf_-(\ell,x)}{z-\lambda_0+\sqrt{\ell_1^2+\ell_2^{2m}}}+\psi(\lambda_+(\ell))\frac{pf_+(\ell,x)}{z-\lambda_0-\sqrt{\ell_1^2+\ell_2^{2m}}}\right)d\ell.
\end{align*}
For simplicity let us choose $\psi=1_{[\lambda_0-\alpha,\lambda_0+\alpha]}$.
Let us first consider the case $m=1$ and use polar coordinates
$$
u_1(x,z)=\int_{S^1}\int_0^{\alpha}e^{i\rho u\cdot x}\left(\frac{pf_-(\rho u,x)}{z-\lambda_0+\rho}+\frac{pf_+(\rho u,x)}{z-\lambda_0-\rho}\right)\rho d\rho du.
$$
Expanding $\rho/(z-\lambda_0\pm\rho)=\pm 1\mp \frac{z-\lambda_0}{z-\lambda_0\pm\rho}$ gives
$$
u_1(x,z)=\int_{S^1}\int_0^{\alpha}e^{i\rho u\cdot x}(pf_-(\rho u,x)-pf_+(\rho u,x))\,d\rho du + z\int_{S^1}\int_0^{\alpha}e^{i\rho u\cdot x}\left(-\frac{pf_-(\rho u,x)}{z-\lambda_0+\rho}+\frac{pf_+(\rho u,x)}{z-\lambda_0-\rho}\right) d\rho du
$$
The stationary phase method shows that the first term belongs to $H^1(\RR^2)$ (exactly as for $w_0$).

Turning to the second term from~\cite{KP}, Theorem 7.2 $pf_{\pm}$ are arc analytic functions so there is a function $qf(\rho, u,x)$ analytic with respect to $\rho$ in a neighborhood of $0$ such that
$$
pf_{\pm}(\pm\rho u,x)=qf(\pm \rho, u, x).
$$
Taking the opposite of $u$ in the integral involving $pf_-$ and then expressing $pf_{\pm}$ in terms of $qf$ one finds that the second term in the previous expression of $u_1(x,z)$ reads
$$
u_1(x,z)=z\int_{S^1} \left (\int_0^{-\alpha}g(\rho,u)d\rho+\int_0^{\alpha}g(\rho,u)d\rho\right)du,\qquad\qquad g(\rho,u)=e^{i\rho u\cdot x}\frac{qf(\rho, u,x)}{z-\lambda_0-\rho}
$$
This expression defines a function of $z\log(z)$ because the integrals in $\rho$ do not combine into one integral running through $0$ (similarly as in the homogeneous case). So the resolvent is indeed continuous but not holomorphic about $\lambda_0$.

\begin{rem}
  This expansion differs from the homogeneous case ($P=-\Delta$) for which the resolvent about $0$ expands $R(z)=P_0\log(z)+P_1+O(z\log(z))$ where $P_j$ are finite rank operators in $L^2(\RR^2)$ (see~\cite{JN}). One could wonder why there is such a difference thinking of $0$ as a nodal point of the characteristic manifold of $-\Delta$. This is of course because the characteristic manifold is the set of $(\omega,\xi)$ (where $\xi$ is the Fourier variable) such that $\omega=\pm|\xi|$, hence $\lambda=|\xi|^2$. So $0$ belongs to $\sigma_0$ (see next paragraph).
\end{rem}

Now for $k$ close to $k_0$ the Fermi curve contains a curve $\Cc$ which is a small circle whose curvature is proportional to $1/|k-k_0|$. In view of formula~\eqref{exp} taking the limit $k\rightarrow k_0$ shows that the leading term whose index is related to $\Cc$ vanishes. Thus if the Fermi level does not intersect any other band then the resolvent about $\lambda_0$ belongs to $O((z-\lambda_0)\log(z-\lambda_0))$.

For $m>1$ odd letting $\ell_2^m=y_2$ we set back to the case $m=1$ up to the jaccobian $|y_2|^{1/m-1}/m$ which is integrable about $0$ even and arc-analytic. So the previous resolvent expansion holds. However the previous consideration about the curvature does not hold since it vanishes in the direction $(1,0)$ so~\eqref{exp} does not hold and one needs to adapt item 2 of theorem~\ref{t2} with a higher order Airy-like function. It is very likely that the $O(1)$ term in the resolvent expansion does not vanish.


\subsection{Generical critical point}\label{Sglancing}

Critical points do not contribute to outgoing waves since the speed vanishes at this point but such points are responsible for some singularity of the resolvent which is logarithmic in the generic 2D case. This situation has been addressed to any dimension by C. Gerard in~\cite{G} Theorem 3.6. However the generality of the quoted paper makes it difficult to understand the way the solution is computed. We thus wish to give a very short calculation to give an insight of the result of~\cite{G} when $\lambda_0\in\sigma_0$ is a non degenerated Morse critical point of $\lambda_n$ at $\ell=0$. In such a case let us write $\lambda_n$ as a quadratic form $\lambda_n(\ell)=\lambda_0+(A(\ell)\ell,\ell)$ where the brackets denote the scalar product in $\RR^2$ and $A$ is a smooth 2 by 2 matrix which does not vanish at $\ell=0$. Let us approximate $\psi$ by the characteristic function of a small neighborhood $V_0$ of $0$. Then choosing $\lambda_{\eps}=\lambda_0+i\eps$ we get
$$
u_{1,\eps}=\int_Be^{i\ell\cdot x}\frac{\psi(\ell)}{\lambda_{\eps}-\lambda_n(\ell)}pf(\ell,x)d\ell=\int_{V_0/\sqrt{\eps}}e^{i\sqrt{\eps}s\cdot x}\frac{1}{i-(A(s\sqrt{\eps})s,s)}pf(s\sqrt{\eps},x)ds.
$$
For $\eps$ small and $0<\alpha<1/2$ the main contribution reads
$$pf(0,x)\int_{V_0/\eps^{\alpha}}\frac{1}{i-(A(0)s,s)}ds.$$
Whatever the signature of $A(0)$ the integral diverges like O$(\ln(\eps))$ as $\eps$ goes to zero. If one replaces $\lambda_{\eps}$ by $\lambda$ in a complex neighborhood of $\lambda_0$ and considers $v_0$ as a function of $\lambda$ one readily sees that $v_0$ is a 2-sheaves analytic function diverging logarithmically at $\lambda_0$. This corresponds to the expression of the outgoing resolvent given in~\cite{G} Corollary 4.2 where one branch of the resolvent expands as follows
$$
(P-\lambda)^{-1}=E_0(\lambda)+C\ln(\lambda-\lambda_0)(M(\lambda_0)+(\lambda-\lambda_0)E_1(\lambda))
$$
where $C$ is a constant, $M$ is a finite dimensional operator and $E_j$ are holomorphic operators about $\lambda_0$.

\section{Greater dimensions}\label{S6}

When the dimension $d$ is greater than 2 the set of points of crossing in a Fermi level is generically of dimension $d-2$. One can still use~\eqref{pw} since the latter set is negligible in the Fermi Level. So if for a Fermi level there is only normal crossing then one can straightforwardly extend Theorem 2.1, 2.2 and 2.3.

From~\cite{KP} the set of points of singular crossing is a $d-2$ dimensional set so we expect $\sigma_1$ to be of non zero Lebesgue mesure and needs to be considered for every $k$. Theorem 7.2 of~\cite{KP} says that in this case there are blowing ups with smooth centers such that crossing eigenvalues become analytic. So theoretically one can compute $u$ as in paragraph~\ref{secnodal}.

\vspace{1cm}
{\Large\bf Appendix}
\appendix

\section{Floquet-Bloch transform and Sobolev spaces}\label{A1}

Let $f\in \cS(\RR^2)$ (Schwarz space). Then its Floquet-Bloch transform defined by
$$
\hat f(x,\ell)=\sum_{n\in\ZZ^2}f(x+2\pi n)e^{-i\ell\cdot(x+2\pi n)}
$$
is periodic with respect to $x$ and quasi-periodic with respect to $\ell$:
$$
\hat f(x,\ell+a_j)=e^{i x\cdot a_j}\hat f(x,\ell),
$$
where $a_j,\ j\in\{1,2\}$ are the unit vectors. Then the inverse Floquet-Bloch transform reads
$$
f(x)=\int_B e^{i\ell\cdot x}\hat f(x,\ell)d\ell.
$$
The Floquet-Bloch transform extends by density to $f\in L^2(\RR^2)$ and Parceval identity holds so that $\hat f\in L^2(W\times B)$ and
$$
\|f\|_{L^2(\RR^2)}=\|\hat f\|_{L^2(B\times W)}.
$$
Then the identity $\widehat{(\D_{x_j}f)}(x,\ell)=(\D_{x_j}+i\ell_j)\hat f(x,\ell)$ for $j\in\{1,2\}$ entails that there is a constant $c>1$ such that
$$
\frac{1}{c} \|f\|_{H^n(\RR^2)}\leq \|\hat f\|_{L^2(B,H^n(W))}\leq c\|f\|_{H^n(\RR^2)}.
$$
Here we used the notation
$$
L^2(B,H^n(W)):=\{h\in L^2(W\times B),\ \D_{\ell}^jh\in L^2(W\times B),\ j\in\{1,\ldots,n\}\}.
$$
Finally the identity $\widehat{(x_jf)}(x,\ell)=\D_{\ell_j}\hat f(x,\ell)$ for $j\in\{1,2\}$ implies
$$
\||x|^n f\|_{L^2(\RR^2)}\leq c \|\hat f\|_{L^2(W;H^n(B))},
$$
where $L^2(W;H^n(B))$ is defined in a similar way as $L^2(B;H^n(W))$.

\section{Analyticity of the Bloch variety}\label{A2}

Let us recall briefly Kuchment's results~\cite{Ku} about the analyticity of $\cB$. For elliptic operators of the form $T:=\sum_{|\alpha|\leq 2m}b_{\alpha}(x)\D^{\alpha}$ with smooth coefficients $b_{\alpha}$, Theorem 3.1.7 (\cite{Ku}) shows the analyticity of $\cB$. Thanks to the smoothness of the coefficients one can use a parametrix which allows to show that the operator $T(\ell)$ resulting from the application of the Floquet-Bloch transform to $T$ and acting on $H^2(B)$ is Fredholm with null index.

Here our operator $P(\ell)$ is of divergence form and with non smooth coefficients. We cannot make use of a parametrix as used in~\cite{Ku} but Theorem 3.1.7 mainly requires a holomorphic family of Fredholm operators with null index (see p. 118).
The opertor $P(\ell)$ is Fredholm because it is of compact resolvent type (standard use of Lax-Milgram Lemma and analytic Fredholm theorem) from $H(W)$ to $L^2(W)$ with
$$H(W)=\{u\in H_{per}^1(W),\alpha (\na u+i\ell)\in H_{div}(W)\}$$
where $H_{div}(W)$ is the subspace of $L^2(W)$ of $W$-periodic functions whose divergence belongs to $L^2(W)$.

Then since $P(\ell)$ is symmetric for real $\ell$ its deficiency index is zero.

The proof of Theorem 3.1.7 needs to be changed because $P(\ell)$ does not map Sobolev spaces to themselves. Moreover the domain of $P(\ell)$ (as operator) depends on $\ell$ since it requires that $\alpha(\na+i\ell)$ belongs to $H_{div}(W)$. We thus consider $p(\ell)$ the associated sesquilinear form on $H^1(W)$ which we expand:
$$
p(\ell)(u,v)=(\alpha\na u,\na v)_2+i\ell\left\{(\alpha u,\na v)_2-(\alpha\na u,v)_2\right\}+\ell^2(\alpha u,v)_2,\qquad u,v\in H^1(W).
$$
We decompose it as a sum of three forms $p(\ell)=p_0+i\ell p_1+(\ell^2-1)p_2$ with
$$
p_0(u,v)=(\alpha\na u,\na v)_2+(\alpha u,v)_2,\quad p_1(u,v)=(\alpha u,\na v)_2-(\alpha\na u,v)_2,\quad p_2(u,v)=(\alpha u,v)_2.
$$
Since $p_1$ is not sectorial on $L^2(W)$ we consider those forms as bounded forms on $H_{per}^1(W,\alpha)$. By Riesz-lemma there are associated operators $P_0,P_1$ and $P_2$ defined on $H_{per}^1(W,\alpha)$ endowed with the scalar product $p_0$. Those operators thus satisfy $p_0(P_ju,v)=p_j(u,v)$. In particular $P_0=Id$. Then denoting by $\tilde P(\ell)$ the operator associated to $p(\ell)$ we have
$$
\tilde P(\ell)=I+i\ell P_1+(\ell^2-1)P_2\quad\text{and observe that}\ ker(P(\ell))\subset ker(\tilde P(\ell)).
$$
Thus it is enough to prove Theorem 3.1.7 replacing $P(\ell)$ by $\tilde P(\ell)$. One thus only needs to prove that $P_j$ are Schatten class operators on $H_{per}^1(W,\alpha)$. This is greatly eased by making use of the operator $S^2:=(P(0)+\alpha)^{-1}$ for which there holds
$$
p_0(S^2u,v)=p_0(Su,Sv)=(u,v)_2,\quad\text{and}\quad P_2=S^2\alpha.
$$
\begin{lem}
  $S$ is a Schatten class operator in $H_{per}^1(W,\alpha)$.
\end{lem}
\begin{proof}
  First $S$ is compact because if we take a bounded sequence $u_n\in H_{per}^1(W,\alpha)$ then by Rellich theorem there is a subsequence (still denoted by $u_n$) such that $u_n$ converges to some $u$ in $L^2(W)$. Then let us show that $Su_n$ converges in $H_{per}^1(W,\alpha)$. We have
  $$
  \| S(u_n-u)\|^2_{H_{per}^1(W,\alpha)}=p_0(S(u_n-u),S(u_n-u))=\|u_n-u\|_2^2
  $$
  which converges to zero.

  Then showing that $S$ is a Schatten class operator can be achieved by comparing $S$ to $\tilde S:=(1-\Delta)^{-1/2}$ which is so in $H_{per}^1(W,1)$. This a simple use of minmax principle.
\end{proof}
As a consequence $P_2$ is a Schatten class operator.
\begin{lem}
$P_1$ is a Schatten class operator of order $p>2$.
\end{lem}
\begin{proof}
  Let us split $P_1=P_{11}-P_{12}$ with
$$
p_{11}(u,v)=(\alpha u,\na v)_2,\quad\text{and}\quad p_{12}(u,v)=(\alpha\na u,v)_2
$$
one cannot easily represent $P_{11}$ and show that it is compact. However we remark that $P_{11}=P_{12}^*$. Indeed
$$
p_0(P_{11}u,v)=p_{11}(u,v)=(\alpha u,\na v)_2=\bar p_{12}(v,u)=p_0(u,P_{12}v).
$$
Then remarking that $P_{12}=S^2\alpha\na$ one sees that it is a Schatten class operator since $S\alpha\na$ is bounded on $H_{per}^1(W,\alpha)$. Then so is $P_{11}$ by duality.

\end{proof}

\section{Estimates in Bloch space}\label{A3}

The following lemma gives an upper bound for the slope of the real Bloch variety. The calculation is reminiscent of geometric optics~\cite{DR} and links the gradient of the band function with the metric. It proves in particular that nodal points of the Bloch variety are not cusps.
\begin{lem}\label{lbband}
For all $n>1$
\begin{equation}
\|\na\sqrt{\lambda_n}\|_{L^{\infty}(B)}\leq \max\alpha/\min\beta.
\end{equation}
\end{lem}

\begin{proof}
Let us expand $P(\ell)$:
$$
\beta P(\ell)=(\div+i\ell)\alpha(\na+i\ell)=\div\alpha\na+i\ell(\na\alpha+\alpha\na)-|\ell|^2.
$$
Setting $P_1=\na\alpha+\alpha\na$ and then differentiating the relation $P(\ell)e_n(\ell)=-\beta\lambda_n(\ell)e_n(\ell)$ with respect to $\ell$ we get
$$
(P(\ell)-\lambda_n(\ell))\na_{\ell} e_n(\ell)+((iP_1+2\ell\alpha)+\beta\na\lambda_n(\ell))e_n(\ell)=0.
$$
Scalarly multiplying by $e_n(\ell)$ which is orthogonal to the first term and a unit vector of $L^2(W)$ gives:
\begin{equation}\label{speed}
\na\lambda_n(\ell)(\beta e_n,e_n)=-((iP_1+2\ell\alpha)e_n,e_n)=2\Im(\alpha(\na+i\ell) e_n,e_n)
\end{equation}
Half of the last factor is bounded by $\|\alpha(\na+i\ell) e_n\|_2$ which is bounded by $\alpha$ times $\|\sqrt{\alpha}(\na+i\ell) e_n\|_2=\sqrt{(P(\ell)e_n,e_n)}=\sqrt{\lambda_n}$.
\end{proof}

\section{Scatering expansion under glancing/grazing Bloch vector}

Pseudo-differential calculus has been carried out by Melrose and Taylor (see~\cite{MeT}). Here, in this basic scattering approach we only need~\cite{Hor} for adapatation of the stationary phase theorem.

Let us resume notations from begining paragraph 5. Denote by $\ell_{t*}=\D_t\ell(t_*)$ and let $\delta\theta$ be the angle between $x$ and the orthogonal of $\ell_{t*}$. Also let $\vec u=x/|x|$ and denote by $\vec u_*=\vec u(\delta\theta=0)$. Finally put $\delta t=t-t_*$. Taylor expansion of $\ell$ with integral remainder about $t_*$ leads  
$$
\vp(\delta t,\delta\theta)=\ell_*\cdot\vec u+\vec u\cdot\vec \ell_{t*}(\delta t+a(\delta t)^2)+(\delta t)^3\vec u\cdot\vec g(\delta t),\quad \vec u_*\cdot\vec g(0)\neq 0.
$$
The coefficient $a$ is a scale factor between $\D_t\ell(t_*)$ and $\D_t^2\ell(t_*)$ since they are collinear. When $\delta\theta=0$ there holds $\vec u_*\cdot\vec \ell_{t*}=0$ so using the variable $T_*$ given by the substitution $T_*=\delta t\sqrt[3]{3\vec u_*\cdot\vec g(\delta t)}$ the phase reads
$$
\vp(\delta t,\delta\theta=0)=\ell_*\cdot\vec u+\frac{T_*^3}{3}.
$$

From~\cite{Hor} there is a change of variable $T=T(T_*,\delta\theta)$ with $T= T_*$ for $\delta\theta=0$ such that the phase reads
$$
\vp(\delta t,\delta\theta)=a(\theta)+b(\theta)T+\frac{T^3}{3}\quad\text{with}\quad a(\theta_*)=\ell_*\cdot\vec u_*,\quad\text{and}\quad b(\theta_*)=0
$$
To find $T,a,b$ just write
$$
\vp(\delta t,\theta)=\ell_*\cdot\vec u+(\vec u\cdot\vec \ell_{t*})T_*h(T_*)+m(\theta,T_*)\frac{T_*^3}{3},\quad m(\theta,T_*)=\frac{\vec u\cdot\vec \tilde{g}(T_*)}{\vec u_*\cdot\vec \tilde{g}(T_*)}
$$
where $\sqrt[3]{3\vec u_*\cdot\vec g(0)}h(0)=1$. Taking $a=\ell_*\cdot\vec u$ and $b=\vec u\cdot\vec \ell_{t*}h(0)$ we see that $T$ is solution of a third degree polynomial equation which has a unique solution for $T\approx T_*$ small since
$$
T\left(b(\theta)+\frac{T^2}{3}\right)=T_*\left((b(\theta)\frac{h(T_*)}{h(0)}+m(\theta,T_*)\frac{T_*^2}{3}\right).
$$
Hence $dT/dT_*(T_*=0)=1$.

Using this substitution in the integral $g_+$, and following Hormander's proof of Theorem 7.7.18 one finds
\begin{equation}\label{Ai1}
g_+=2\pi\frac{e^{ix\cdot\ell_*}}{r^{1/3}}\left ( Ai\left( \frac{x\cdot\ell_{t*}}{r^{1/3}\sqrt[3]{3\vec u_*\cdot\vec g(0)}}\right) w_*(x) + \frac{1}{r^{1/3}}Ai'\left( \frac{x\cdot\ell_{t*}}{r^{1/3}\sqrt[3]{3\vec u_*\cdot\vec g(0)}}\right) \tilde w_*(x)\right ) + O\left(\frac{1}{r^{3/2}}\right)
\end{equation}
where $w_*$ (resp. $\tilde w_*$) is the zeroth (resp. first) order Taylor expansion in the variable $T$ of the integrand in $g_+$ (after $t\rightarrow T$ substitution). One thus finds
\begin{align}\label{Ai2}
& w_*(x)=\sqrt[3]{3\vec u_*\cdot\vec g(0)}(\hat f(\ell_*),v(\ell_*))_{L^2(W)}\frac{|\D_t\ell(t_*)|}{|\na\mu(\ell_*)|}v(\ell_*,x)\\
&\tilde w_*(x)=\frac{1}{i}\frac{d}{dt}\left((\hat f(\ell),v(\ell))_{L^2(W)}\frac{|\D_t\ell|}{|\na\mu(\ell)|}\frac{d\delta t}{d T}v(\ell,x)\right)(t=t_*)\nonumber
\end{align}

\bibliographystyle{plain}
\bibliography{bibli.bib}


\end{document}